\documentclass[EJP]{ejpecp} 

\usepackage{xcolor} 
\usepackage[normalem]{ulem} 

\SHORTTITLE{Explosion of CMJ processes} 

\TITLE{Explosion of Crump-Mode-Jagers processes with critical immediate offspring.} 

\AUTHORS{Gerold Alsmeyer\footnote{Universit\"at M\"unster, Germany. \EMAIL{gerolda@uni-muenster.de}} \and Konrad Kolesko\footnote{Uniwersytet Wroc\l awski, Poland. \EMAIL{konrad.kolesko@math.uni.wroc.pl}} \and Matthias Meiners\footnote{Justus-Liebig-Universit\"at Gie\ss en, Germany. \EMAIL{matthias.meiners@math.uni-giessen.de, jakob.stonner@math.uni-giessen.de}} \and Jakob Stonner\footnotemark[3]} 


\AMSSUBJ{60J80 ; 39B22 ; 60G55} 

\SUBMITTED{February 13, 2025} 

\VOLUME{0} 
\YEAR{2023} 
\PAPERNUM{0} 
\DOI{10.1214/YY-TN} 

\ABSTRACT{We study the phenomenon of explosion in general (Crump-Mode-Jagers) branching processes, which refers to the event where an infinite number of individuals are born in finite time. In a critical setting where the expected number of immediate offspring per individual is exactly one, whether or not explosion occurs depends on the fine properties of the reproduction point process. We provide two sufficient conditions for explosion in these CMJ processes. The first uses a comparison with Galton-Watson processes in varying environments, while the second relies on a comparison with Bellman-Harris branching processes. Our main result is an equivalent characterization of explosion, expressed as an integral test, in the case where the reproduction point process is Poisson. 
For the derivation, we also study the fixed-point equation associated with a smoothing transform, which is known to describe the distribution of the explosion time. 
We use multiplicative martingales to show that this distribution is an attractive fixed point of the smoothing transform, which in particular implies its uniqueness modulo an additive shift.} 

\newcommand{\N}{\mathbb{N}}

\newcommand{\R}{\mathbb{R}}

\newcommand{\cZ}{\mathcal{Z}} 

\newcommand{\F}{\mathcal{F}}

\newcommand{\cL}{\mathcal{L}} 
\newcommand{\cM}{\mathcal{M}}

\newcommand{\Prob}{\mathbf{P}}

\newcommand{\E}{\mathbf{E}}

\newcommand{\eqdist}{%
\mathrel{\vbox{\offinterlineskip\ialign{%
\hfil##\hfil\cr 
$\scriptscriptstyle\mathrm{d}$\cr 
\noalign{\kern.1ex} 
$=$\cr 
}}}}

\newcommand{\1}{\mathbb{1}}

\newcommand{\distto}{%
\mathrel{\vbox{\offinterlineskip\ialign{%
\hfil##\hfil\cr 
$\scriptscriptstyle\mathrm{d}$\cr 
\noalign{\kern-.05ex} 
$\to$\cr 
}}}} 
\newcommand{\Probto}{%
\mathrel{\vbox{\offinterlineskip\ialign{%
\hfil##\hfil\cr 
$\scriptscriptstyle\Prob$\cr 
\noalign{\kern-.05ex} 
$\to$\cr 
}}}} 
\newcommand{\TVto}{%
\mathrel{\vbox{\offinterlineskip\ialign{%
\hfil##\hfil\cr 
$\scriptscriptstyle\mathrm{TV}$\cr 
\noalign{\kern-.05ex} 
$\to$\cr 
}}}}

\newcommand{\dx}{\mathrm{d} \mathit{x}} 
\newcommand{\dy}{\mathrm{d} \mathit{y}}

\newcommand{\e}{\mathrm{e}}

\newcommand{\defeq}{\vcentcolon=} 
\newcommand{\eqdef}{=\vcentcolon} 
\newcommand{\eps}{\varepsilon} 
\renewcommand\d{\mathrm{d}} 

\newcommand{\I}{\mathcal{I}}

\DeclareMathOperator{\Smooth}{\mathbf{T}} 
 
\newcommand\Fbar{\overline{F}} 

\begin{document} 



\section{Introduction} 

In this work, we investigate the phenomenon of explosion in general branching processes, which here refers to the possibility that the process exhibits infinitely many birth events within a finite time, with positive probability. A general branching process, also known as a Crump-Mode-Jagers (CMJ) process, can be described as follows. The population begins with a single individual, the ancestor, who produces offspring at random nonnegative times, forming a point process $\xi$ on the nonnegative real line. Each new individual then produces their own offspring, and the point process representing the birth times of the children of each individual is an independent, identically distributed (i.i.d.) copy of $\xi$, shifted by the birth time of the parent. A precise definition of the model is given in Section \ref{subsec:General branching process}. 

\vspace{.1cm} 
The concept of explosion in CMJ processes plays a key role in various contexts, such as preferential attachment trees \cite{Iyer:2024}, the configuration model \cite{Dereich+Ortgiese:2018}, weighted random graphs \cite{Baroni+al:2017}, and doubly stochastic Yule cascades \cite{Dascaliuc+al:2023}, see also \cite{Komjathy:2016} and the references therein. 

\vspace{.1cm} 
In many applications, the existence of a \emph{Malthusian parameter} is assumed, a technical assumption which ensures exponential growth of the expected number of individuals over time, rendering the process finite almost surely for all times \cite{Jagers:1975,Nerman:1981}. Additionally, if the number of instantaneous births $\xi\{0\}$ satisfies $\E[\xi\{0\}] < 1$ and $\E[\xi[0,\eps]] < \infty$ for some $\eps > 0$, the process remains finite at all times \cite[Theorem (6.2.2)]{Jagers:1975}. On the other hand, if $\E[\xi\{0\}] > 1$, each birth triggers an instantaneous supercritical Galton-Watson process, leading to immediate explosion \cite[Theorem (6.2.1)]{Jagers:1975}. 

Our primary focus is on the critical case where $\E[\xi\{0\}] = 1$. In this scenario, each birth event triggers the creation of a population equivalent to a critical Galton-Watson process, which is finite almost surely. Whether the branching process exhibits explosion in this critical case is determined by the behavior of $\xi$ near zero. For instance, a classical result by Jagers \cite[Theorem (6.2.3)]{Jagers:1975} asserts that if there exists an $\eps>0$ such that $\E[\xi[0,\eps]] = 1$, the process will stay finite (see \cite[Theorem 3.1]{Komjathy:2016} for a different proof). Jagers also noted that if $\E[\xi\{0\}] = 1$ and $\E[\xi[0,t]]$ grows sufficiently fast, explosion may occur. For a more detailed survey with proofs\footnote{In Corollary 3.3 of \cite{Komjathy:2016}, the author claims without proof that CMJ processes with a reproduction point process that has a locally finite intensity measure cannot explode, which is in contradiction with our results. The other results that we cite from this work appear to be correct.}, see the unpublished preprint \cite{Komjathy:2016}. 

In this article, we prove two sufficient conditions for the explosion of general branching processes (Theorems \ref{Thm:sufficient-explosion-GWVE} and \ref{Thm:sufficient-explosion-Amini}). We focus particularly on the case where the reproduction process $\xi$ is a Poisson point process. We provide a characterization of explosion in this case, specifically when the cumulative mass function of the intensity measure is convex near zero (Theorem \ref{Thm:explosion-Poisson}). Additionally, we present comparison results that allow us to deduce the explosion of one CMJ process with a Poisson reproduction process from the explosion of another, by comparing their intensity measures (see Section \ref{subsec:comparison}). 

Most existing results on the explosion behavior of branching processes have been derived for \emph{Bellman-Harris processes}. 
The latter, also known as age-dependent branching processes, correspond to the case where $\xi = Z \delta_{W}$, with independent random variables $Z$ and $W$ taking values in $\N_0$ and $[0,\infty)$, respectively. Key contributions in this area include the works by Sevast'yanov \cite{Sevastyanov:1970,Sevastyanov:1971}, Vatutin \cite{Vatutin:1976,Vatutin:1993}, Grey \cite{Grey:1974}, and Grieshechkin \cite{Grishechkin:1987}. Many of these results have direct counterparts for general branching processes, as discussed in \cite{Komjathy:2016}. Significant progress was made in \cite{Amini:2013}, where explosion for Bellman-Harris processes was characterized under a heavy-tail assumption on the offspring distribution, which roughly corresponds to $\Prob(Z > n) \sim n^{-\alpha}$ for some $\alpha\in(0,1)$. One of our sufficient criteria for explosion is derived from this result. 

The explosion phenomenon is characterized by the existence of a nontrivial solution within the set of non-increasing, left-continuous functions $\phi: \mathbb{R} \to [0,1]$ satisfying $\phi(t) = 1$ for $t < 0$, to the functional equation of the smoothing transform 
\begin{equation}\label{eq:Fbar-fixed-point} 
\phi(t) = \E\bigg[\prod_{j=1}^N \phi(t - X_j) \bigg],\quad t \geq 0 
\end{equation} 
where the reproduction point process is written as $\xi = \sum_{j=1}^N \delta_{X_j}$. 
By the branching property of CMJ processes, the left-continuous version of the survival function $\Fbar(t) \coloneqq \Prob(T \geq t)$ of the explosion time $T$ satisfies \eqref{eq:Fbar-fixed-point}, i.e., $\phi=\Fbar$ is a solution. Moreover, explosion occurs if and only if this equation admits a \emph{non-trivial solution}, meaning a function $\phi$ with $\phi(t) < 1$ for some $t > 0$. Equation~\eqref{eq:Fbar-fixed-point} plays a fundamental role in the analysis of explosion and has been studied extensively in earlier works, including \cite{Grey:1974,Grishechkin:1987,Sevastyanov:1970,Sevastyanov:1971,Vatutin:1976,Vatutin:1993}; see also \cite[Section 2]{Komjathy:2016}. 
In this paper we prove that $\Fbar$ is an \emph{attractive} fixed point of the smoothing transform, in the sense that, that for all suitable functions $\phi$, the iterates $\Smooth^n \! \phi$ converge pointwise to $\Fbar$ as $n \to \infty$, where $\Smooth \! \phi(t)$ denotes the right-hand side of \eqref{eq:Fbar-fixed-point} (Theorem \ref{thm:uniqueness}). This result was previously known only in the special case of Bellman-Harris processes under the additional assumption that the displacement distribution has a continuous density at zero~\cite{Grishechkin:1987}. 
In particular, our result implies that $\phi=\Fbar$ is the unique non-trivial solution to \eqref{eq:Fbar-fixed-point}, up to a shift in the argument. This uniqueness had been established for CMJ processes only under the additional assumption $\Prob(\xi[0,\infty)<\infty) = 1$; see \cite[Corollary 5.2]{Jagers+Roesler:2004}. Our approach relies on the construction of an associated multiplicative martingale, which also yields, along the way, new proofs of several classical results concerning explosion. 

\section{Assumptions and main results} 
\subsection{Crump-Mode-Jagers branching processes} \label{subsec:General branching process} 
We formally introduce the Crump-Mode-Jagers (CMJ) branching process, also known as the general branching process, following the framework of Jagers \cite{Jagers:1975,Jagers:1989}. 

\vspace{.1cm} 
The process starts with a single individual, the ancestor, born at time $0$. This individual gives rise to offspring at times governed by a point process on $[0,\infty)$, called the \emph{reproduction point process}, which we denote by $\xi$. 
Here and throughout the paper, we assume that $\xi$ is a random, \emph{locally finite} point measure, meaning that $\xi(t) \defeq \xi[0,t] < \infty$ almost surely for all $t \geq 0$. 
We may therefore write $\xi = \sum_{j=1}^N \delta_{X_j}$ 
where $N = \xi[0,\infty)$ denotes the total number of offspring and takes values in $\N_0 \cup \{\infty\}$ with $\N_0 = \{0,1,2,\ldots\}$. For each $j$, the birth time $X_j$ of the $j$-th offspring is defined by 
\begin{align*} 
X_j \defeq \inf\{t \geq 0: \xi[0,t] \geq j\} 
\end{align*} 
where, by convention, the infimum of the empty set is taken to be $\infty$. 

\vspace{.1cm} 
We denote the intensity measure of the process by $\mu$, i.e., $\mu(B) = \E[\xi(B)]$ for every Borel set $B \subseteq \R$. For the sake of readability, we occasionally abbreviate $\mu(B)$ as $\mu B$ and $\xi(B)$ as $\xi B$, as we have done earlier. 

\vspace{.1cm} 
Individuals in the process are labeled using Ulam-Harris notation, i.e., by finite sequences (or words) of positive integers. Formally, the set of all labels is given by $\I = \{\varnothing\} \cup \bigcup_{n \in \N} \N^n$, where $\N = \{1, 2, \ldots\}$, and $\varnothing$ denotes the empty word, corresponding to the ancestor. We abbreviate a word $u = (u_1, \ldots, u_n) \in \N^n$ as $u_1 \ldots u_n$, and refer to $|u| \coloneqq n$ as the length or, equivalently, the generation of $u$. 
Each word $u\in I$ represents a potential individual in the branching process, and its ancestral lineage is encoded as the chain 
\begin{equation*} 
\varnothing \to u_1 \to u_1 u_2 \to \ldots \to u_1 \ldots u_n = u. 
\end{equation*} 
For two labels $u=u_1\ldots u_n$ and $v = v_1 \ldots v_m \in \I$, the concatenation $uv$ denotes the individual labeled $u_1 \ldots u_n v_1 \ldots v_m$. For $u \in \I$, we call the individuals $\{ui\,:\,i \in \N\}$ the children of $u$. Conversely, for any $i \in \N$, we call $u$ the parent of $ui$. 

For each $u \in \I$ there is an independent copy $\xi_u$ 
of $\xi$ that determines the birth times of $u$'s offspring relative to its time of birth. 
The birth times $S(u)$ for $u \in \I$ are defined recursively. 
We set $S(\varnothing) \defeq 0$ and, for $n \in \N_0$, 
\begin{equation}\label{eq:S-def} 
S(uj) \defeq S(u) + X_{u,j} \quad \text{for } u \in \N^n \text{ and } j \in \N, 
\end{equation} 
where $\xi_u = \sum_{j}\delta_{X_{u,j}}$. 
To simplify notation, we also write $S_u$ for $S(u)$. 
The number of individuals born up to and including time $t \geq 0$ is denoted by 
\begin{equation}\label{eq:Z-def} 
\cZ_t \defeq \#\{u \in \I: S(u) \leq t\}. 
\end{equation} 
The \emph{explosion time} of the Crump-Mode-Jagers process is defined as 
\begin{equation}\label{eq:explosion-time} 
T \defeq \inf\{t \geq 0: \cZ_t = \infty\} \in [0,\infty]. 
\end{equation} 

\begin{definition} \label{Def:explosive} 
The CMJ process is called \emph{explosive} if $\Prob(T < \infty) > 0$, where $\{T < \infty\}$ is referred to as the \emph{explosion event}. Otherwise, the process is called \emph{non-explosive}. 
We say that a reproduction point process $\xi$ \emph{yields an explosive CMJ process} if the CMJ process driven by $\xi$ is explosive. 
\end{definition} 

By \cite[Proposition 3.5]{Iyer:2024}, conditional on survival, explosion occurs with probability $0$ or $1$. 
Our assumption that $\xi[0,t]<\infty$ almost surely for all $t > 0$ rules out explosion caused by a single individual producing infinitely many offspring in finite time -- an event referred to as \emph{sideways explosion} in \cite{Komjathy:2016}. In our setting, explosion occurs if and only if every generation contains at least one individual born before some fixed time. In fact, given this, an application of K\H{o}nig's lemma shows that explosion implies the existence of an \emph{infinite line of descent} along which all individuals are born before a fixed time. 
To see this, we define the time of the first birth in generation $n$ by 
\begin{equation}\label{eq:minimum} 
M_n \coloneqq \min_{|u| = n} S(u),\quad n \in \N_0 
\end{equation} 
and verify that (cf.~\cite[Lemma 1.3]{Komjathy:2016}) 
\begin{equation}\label{eq:T-approx} 
\lim_{n \to \infty} M_n = T \quad \text{almost surely.} 
\end{equation} 
Indeed, the limit $M_{\infty}:=\lim_{n \to \infty} M_n$ exists almost surely in $[0,\infty]$ by monotonicity. Moreover, for any $t\ge 0$, $M_{\infty}\leq t$ implies that there exists an infinite line of descent with all individuals born no later than time $t$, and thus $T \leq t$ almost surely. This gives $T\le M_{\infty}$. 
Conversely, if $T < t$, then $\cZ_t = \infty$, i.e., the number of individuals born before time $t$ is infinite. However, since $\#\{u \in \I\,:\, |u| \leq n,\,S(u) \leq t\}$ is finite for all $n \in \N$, there exists a sequence of individuals $(u_n)_{n \in \N}$ such that $|u_n| = n$ and $M_{n}\leq S(u_n) \leq t$ for all $n \in \N$. 
This shows $M_{\infty}\leq t$, so that we conclude \eqref{eq:T-approx}. 

\subsection{Preliminaries} 
The goal of this paper is to establish sufficient and necessary conditions on the reproduction point process $\xi$ for explosion of the associated CMJ process. 
While the two sets of conditions differ in general, they coincide in certain specific cases. 
Our starting point is the following result, stated in Jagers' classical textbook \cite[Theorems (6.2.1) through (6.2.3)]{Jagers:1975}, which forms the foundation for our analysis. 
For brevity, we write $\mu(t) = \mu[0,t]$ and $\xi(t) = \xi[0,t]$ for $t \geq 0$. 

\begin{proposition} \label{Prop:Jagers} 
Let $\xi$ be a reproduction point process with intensity measure $\mu$. 
\begin{itemize} 
\item If $\mu(0) > 1$ or $\Prob(\xi(0)=1) = 1$, then $\xi$ yields an explosive CMJ process. 
\item Conversely, if $\mu(0) < 1$ and $\mu(\eps) < \infty$ for some $\varepsilon > 0$, or $\mu(\eps) = 1$ for some $\eps > 0$ and $\Prob(\xi(0)=1) < 1$, 
then $\xi$ does not yield an explosive CMJ process. 
\end{itemize} 
\end{proposition} 

In view of Proposition \ref{Prop:Jagers}, additionally to the standing assumption $\xi(t) < \infty$ almost surely for all $t \geq 0$, 
we assume throughout that $\Prob(\xi(0)=1) < 1$. 
Based on these assumptions, it remains to examine the following two cases with regard to explosion. 
\begin{gather*} \tag{A0} \label{eq:critical} 
\mu(0) = 1 
\shortintertext{and} 
\tag{A0'} \label{eq:immediately infinite intensity} 
\mu(0) < 1 \quad \text{and} \quad \mu(t) = \infty \quad \text{for all } t > 0. 
\end{gather*} 
This paper focuses on the first case, so let us assume that \eqref{eq:critical} holds. For partial results related to the case where \eqref{eq:immediately infinite intensity} holds, see e.g.\ \cite[Theorem IV.2.1]{Asmussen+Hering:1983}, \cite{Amini:2013} and \cite{Iyer:2024} (cf.\ Remark \ref{rem:Iyer-comparison}). 

\begin{remark} 
It is possible to reduce the setting described by \eqref{eq:critical} to that of \eqref{eq:immediately infinite intensity} using a construction due to Bramson~\cite{Bramson:1978}, which has been employed in several subsequent works, including~\cite{Amini:2013,Komjathy:2016}. 
Let $(S(u))_{u \in \I}$ denote the birth times in a CMJ process with reproduction point process $\xi$, satisfying \eqref{eq:critical}. Define the set of \emph{instantly born individuals} by 
\begin{equation*} 
\I_0 \coloneqq \{u \in \I\,:\, S(u) = 0\}. 
\end{equation*} 
The number of such individuals, $\# \I_0$, coincides with the total progeny of a critical Galton-Watson process with offspring distribution $\xi(0)$, and is therefore finite almost surely. Now construct a new point process 
\begin{equation*} 
\xi' \coloneqq \sum_{u \in \I_0} \xi_u(\cdot \cap (0,\infty)) 
\end{equation*} 
where $\xi_u$ denotes the copy of $\xi$ governing the reproduction times of the children of $u \in \I$, defined relative to the birth time $S(u)$, as in Section \ref{subsec:General branching process}. 
The process $\xi'$ satisfies \eqref{eq:immediately infinite intensity}, and the CMJ process driven by 
$\xi'$ has the same explosion behavior as the original one with $\xi$; see \cite[Theorem 3.1]{Komjathy:2016}. 
Thus we see that condition \eqref{eq:critical} can be reduced to condition \eqref{eq:immediately infinite intensity}. 
Nevertheless, the available results in the literature are not directly applicable, which motivates us to investigate this case within the present paper. 
Further, we will make use of this reduction in the proof of Theorem~\ref{Thm:sufficient-explosion-Amini}, specifically in Lemma~\ref{lemma:Ydelta_0+delta_W}. 
\end{remark} 
Let $\mu_+(\cdot) \defeq \mu(\cdot \cap (0,\infty))$ denote the restriction of $\mu$ to the strictly positive half-line, and define 
\begin{equation*} 
\mu_+(t) \defeq \mu(t) - 1,\quad t \geq 0, 
\end{equation*} 
where we use $\mu_+$ to refer both to the measure and its associated cumulative mass function, with slight abuse of notation. We impose the following assumptions: 
\begin{gather} \tag{A1} \label{eq:finitemu_+} 
\mu_+(t_0) < \infty \quad \text{for some } t_0 > 0, 
\shortintertext{and} 
\tag{A2} \label{eq:positivemu_+} 
\mu_+(t) > 0 \quad \text{for all } t > 0. 
\end{gather} 
Note that, under the critical condition \eqref{eq:critical}, assumption \eqref{eq:positivemu_+} is necessary for explosion; if it fails, the CMJ process is non-explosive by Proposition~\ref{Prop:Jagers}. 
Finally, we define the generalized inverse of $\mu_+$ as 
\begin{equation*} 
\mu_+^{-1}(y)\,\coloneqq\, \inf\{x \geq 0:\,\mu_+(x) \geq y\},\quad y \geq 0. 
\end{equation*} 

\subsection{Explosion results} 
Our first result provides sufficient conditions for $\xi$ to yield an explosive CMJ process. 
\begin{theorem} \label{Thm:sufficient-explosion-GWVE} 
Let $\xi$ be a point process satisfying \eqref{eq:critical}, \eqref{eq:finitemu_+}, \eqref{eq:positivemu_+} and $\E[\xi(\eps)^2]<\infty$ for some $\eps > 0$. Assume further that 
\begin{equation} \label{eq:sufficient explosive general} 
\liminf_{t \downarrow 0} \frac{\mu_+(t)}{t |\log t|^{1+\delta}} > 0\quad\text{for some }\delta>0. 
\end{equation} 
Then $\xi$ yields an explosive CMJ process. 
\end{theorem} 

\begin{remark}\label{rem:Iyer-comparison} 
(i) It is suggestive to compare Theorem \ref{Thm:sufficient-explosion-GWVE} to the sufficient condition given by Iyer \cite[Corollary 3.11]{Iyer:2024}, which states that $\xi$ is explosive if there exists $\eps > 0$ such that 
\begin{equation*} 
\Prob(\xi(t) > x)\,>\,\frac{(\log x)^{1 + \eps}t}{x} 
\end{equation*} 
for all sufficiently small $t > 0$ and sufficiently large $x$. At first glance, this condition resembles the one we give in \eqref{eq:sufficient explosive general}. However, Iyer's condition actually implies that $\xi$ has infinite intensity. To see this, we note that 
\begin{equation*} \E[\xi(t)]\ \geq\ \int_{x_0}^\infty \Prob(\xi(t) > x)\, \d x\ \geq\ t \int_{x_0}^\infty \frac{(\log x)^{1 + \eps}}{x}\, \d x\ =\ \infty 
\end{equation*} 
for some $x_0 > 0$ and all sufficiently small $t > 0$. Therefore, Iyer's criterion does not relate to our condition in \eqref{eq:sufficient explosive general}, which applies under finite intensity. Instead, it should be viewed as a condition appropriate for the setting of \eqref{eq:immediately infinite intensity}. 

\noindent 
(ii) The second moment assumption in Theorem \ref{Thm:sufficient-explosion-GWVE} enters as a requirement necessary to apply a survival criterion for Galton-Watson processes in varying environment (GWVE), see Theorem \ref{thm:kersting} below, and has resisted our truncation attempts. 
A survival criterion for GWVE without a second moment assumption 
would likely lead to a sufficient explosion criterion, 
possibly requiring a stronger condition than \eqref{eq:sufficient explosive general}. 

\noindent 
(iii) Condition \eqref{eq:sufficient explosive general} might be improved. The most general version that our approach allows is provided in Lemma \ref{Lem:explosion via GWPVE}. 
\end{remark} 

While Theorem \ref{Thm:sufficient-explosion-GWVE} offers a broadly applicable sufficient condition for explosion, it does cover the case where the cumulative mass function $\mu_+$ is sublinear at zero. The following result addresses this gap by imposing a weaker condition on $\mu_+$, at the expense of an additional independence assumption. 

\begin{theorem}\label{Thm:sufficient-explosion-Amini} Let $\xi$ be a point process that satisfies \eqref{eq:critical}, \eqref{eq:finitemu_+}, and \eqref{eq:positivemu_+}, along with the following conditions: 
\renewcommand{\labelenumi}{(\roman{enumi})} 
\begin{enumerate} 
\item The instant offspring $\xi(0)$ is independent of ${\xi(\cdot \cap (0,\infty))}$. 
\item The distribution of $\xi(0)$ belongs 
to the domain of attraction of the normal distribution or, more generally, 
a stable law with index $\alpha \in (1,2]$. 
\item There exist constants $t, \delta > 0$ such that 
\begin{equation*} 
\E[\xi(0,t]^{1+\delta}]\,<\,\infty. 
\end{equation*} 
\item There exists $\eps \in (0,1)$ such that 
\begin{equation*} 
\int_0^\eps \frac{\mu_+^{-1}(x)}{x |\log x|} \, \d x\,<\,\infty. 
\end{equation*} 
\end{enumerate} 
Then $\xi$ yields an explosive CMJ process. 
\end{theorem} 

Note that conditions (i), (ii), and (iii) are always met by Poisson point processes. In this case, and assuming a local convexity condition on $\mu_+$, condition (iv) is both necessary and sufficient for $\xi$ to yield an explosive CMJ process. 
\begin{theorem}\label{Thm:explosion-Poisson} 
Let $\xi$ be a Poisson point process that satisfies \eqref{eq:critical}, \eqref{eq:finitemu_+}, and \eqref{eq:positivemu_+}. Furthermore, assume that the cumulative mass function $\mu_+$ is convex in a neighborhood of zero. Then $\xi$ yields an explosive CMJ process if and only if 
\begin{equation}\label{eq:explosion-integral-cond} 
\int_0^\eps \frac{\mu_+^{-1}(x)}{x |\log x|} \, \d x\,<\,\infty. 
\end{equation} 
for some $\eps>0$. 
\end{theorem} 

\begin{remark} 
In Section \ref{subsec:comparison}, we present comparison results that relax the assumptions of Theorem \ref{Thm:explosion-Poisson} (similar to \cite[Section 4]{Grey:1974} for Bellman--Harris processes). For example, Proposition \ref{prop:Poisson-comparison} shows that if there exists a convex function $\varphi: [0, \infty) \to [0,\infty)$ such that 
\begin{equation*} 
\varphi(t)\,\leq\,\mu_+(t)\,\leq\,a \varphi(t) 
\end{equation*} 
for all sufficiently small $t$ and some constant $a > 1$, then condition \eqref{eq:explosion-integral-cond} still characterizes the explosiveness of the CMJ process with Poisson reproduction point process $\xi$. 
\end{remark} 

\subsection{Smoothing transform results} 

The smoothing transform $\Smooth$, associated with the reproduction point process $\xi$, is a self-map on the space $\cM$ of non-increasing, left-continuous functions $\phi: \mathbb{R} \to [0, 1]$ with $\phi(t) = 1$ for all $t \leq 0$. It is defined by 
\begin{align} \label{eq:ST} 
\Smooth \! \phi (t)\, \defeq\, \E\bigg[\prod_{|u|=1} \phi(t - S_u)\bigg], \quad t \in \R. 
\end{align} 
It is straightforward to verify that $\Smooth$ is well-defined and order-preserving, i.e., $\phi \leq \psi$ implies $\Smooth \! \phi \leq \Smooth \! \psi$, for all $\phi, \psi \in \cM$. 
The smoothing transform plays an important role in the theory of branching random walks, see \cite{Alsmeyer+al:2012} for recent results and a comprehensive overview of the literature. 
It also arises naturally in the context of explosion of CMJ processes, as we now explain. 

Let $T_u$ denote the explosion time of the CMJ process within the subtree rooted at individual $u$, defined analogously to $T$ in \eqref{eq:explosion-time}. Observe that $T_u$ has the same distribution as $T$, and that 
\begin{equation*}
T\,=\,\min_{|u|=1} (S_u + T_u) \quad \text{almost surely} 
\end{equation*} 
where the $T_u$ for $|u| = 1$ are mutually independent and also independent of the $S_u$. 

Let $F(t) \defeq \Prob(T < t)$, for $t \in \mathbb{R}$, be the left-continuous distribution function of $T$, and let $\Fbar \defeq 1 - F$ be the corresponding survival function. 
Then, $\Fbar$ satisfies the functional equation of the smoothing transform, i.e., $\Smooth \! \Fbar = \Fbar$. 
In fact, $\Fbar$ is the minimal fixed point of $\Smooth$, a result first established by Grey in the setting of Bellman-Harris processes \cite[Theorem 1]{Grey:1974}, and later extended to general CMJ processes in \cite[Lemma 2.1]{Komjathy:2016}. 
Another fixed point of $\Smooth$ is the constant function $\phi \equiv 1$, which we refer to as the \emph{trivial fixed point}. 

The following theorem is our main contribution to the theory of the smoothing transform. 
\begin{theorem}\label{thm:uniqueness} 
Assume \eqref{eq:critical}. 
Then the following assertions hold: 
\renewcommand{\labelenumi}{(\alph{enumi})} 
\begin{enumerate} 
\item 
If $\phi \in \cM$ is non-trivial and $\phi = \Smooth \! \phi$, 
then $\phi(t) = \Fbar(t-c)$ for all $t \geq 0$ and some $c \geq 0$. 
\item 
Any $\phi \in \cM$ with $\phi(t) < 1$ for all $t > 0$ satisfies 
$\lim_{n \to \infty} \Smooth^n \! \phi(t) = \Fbar(t)$ for any $t \geq 0$. 
\end{enumerate} 
\end{theorem} 

Assertion (a) was previously established in \cite[Corollary 5.2]{Jagers+Roesler:2004} 
under the additional assumption $\#\{|u| = 1\,:\,S_u < \infty\} < \infty$ almost surely. 
Assertion (b) states that $\Fbar$ is the unique attracting fixed point 
among those $\phi \in \cM$ with $\phi(t) < 1$ for all $t > 0$. 
This result was shown for Bellman-Harris processes in \cite{Grishechkin:1987}, assuming additionally that the displacement distribution admits a continuous density at zero. 

The proof of Theorem \ref{thm:uniqueness} also yields new proofs of several known results, including the minimality of $\Fbar$ among fixed points of $\Smooth$ (Proposition \ref{prop:minimal-fixed-point}) and the fact that explosion, when it occurs, can happen arbitrarily fast with positive probability (i.e., $\Fbar(t) < 1$ for all $t > 0$; see Corollary \ref{cor:minimal fixed point}). These statements also follow directly from Theorem \ref{thm:uniqueness}, highlighting its unifying nature. 

We will apply Theorem \ref{thm:uniqueness} in the proof of Theorem \ref{Thm:explosion-Poisson}. 

\subsection{Outline} 

The paper is organized as follows. In Section \ref{sec:smoothing transform}, we present the proof of Theorem \ref{thm:uniqueness}, our main result concerning the smoothing transform. 
In Section \ref{subsec:sufficient-explosion-GWVE}, we prove Theorem \ref{Thm:sufficient-explosion-GWVE} by constructing a Galton-Watson process in a varying environment and showing that its survival implies explosion of the original CMJ process. 
Section \ref{subsec:sufficient-explosion-Amini} uses a result of Amini et al.\ \cite{Amini:2013} and a comparison with Bellman-Harris processes to establish Theorem \ref{Thm:sufficient-explosion-Amini}. 
Section \ref{subsec:explosion-poisson} focuses on Poisson point processes; to prove Theorem \ref{Thm:explosion-Poisson}, we adapt the general approach of Grishechkin \cite{Grishechkin:1987} to our setting. 
Finally, in Section \ref{subsec:comparison}, we introduce comparison techniques for establishing explosion of CMJ processes with Poisson reproduction, based on intensity measure comparisons. 

\section{Proof of Theorem \ref{thm:uniqueness}} \label{sec:smoothing transform} 

We begin by providing a new proof of the fact that $\Fbar$ is the minimal fixed point of the smoothing transform $\Smooth$.  
\begin{proposition}\label{prop:minimal-fixed-point} 
If $\phi \in \cM$ satisfies $\phi = \Smooth \! \phi$, then $\Fbar \leq \phi$. 
In particular, $\Smooth$ has a non-trivial fixed point if and only if the associated CMJ process is explosive. 
\end{proposition} 

Our proof follows as an almost immediate consequence of the next lemma. 
\begin{lemma}\label{lemma:M^phi_n} 
Let $\phi \in \cM$ and define 
\begin{equation}\label{eq:M^phi_n} 
M^\phi_n(t) \defeq \prod_{|u|=n} \phi(t - S_u),\quad n \in \N_0,\ t \geq 0. 
\end{equation} 
Then, for all $t \geq 0$, 
\begin{equation}\label{eq:M-limit=1} 
\Prob\big(M_n^{\phi}(t) = 1\text{ for all but finitely many }n\,\big|\,T>t\big)\,=\,1. 
\end{equation} 
Moreover, if $\phi \in \cM$ satisfies $\Smooth \! \phi = \phi$, then $M^\phi_n(t)$ converges almost surely and in mean to a random variable $M^\phi_\infty(t)$, which satisfies $\E[M^\phi_\infty(t)] = \phi(t)$ and $M^\phi_\infty(t) \geq \1\{T > t\}$ almost surely. 
\end{lemma} 

\begin{proof} 
Let $M_n \defeq \min_{|u|=n} S_u$ be the minimal birth time in the $n$-th generation, $n \in \N_0$. Since $M_n \uparrow T$ as $n \to \infty$ by \eqref{eq:T-approx}, it follows that for sufficiently large $n$, we have $t - M_n < 0$ on the event $\{T > t\}$ and therefore $\phi(t - S_u) \geq \phi(t - M_n) = 1$ for all individuals $u$ with $|u| = n$. This implies \eqref{eq:M-limit=1}. 
Next, observe that if $\phi \in \cM$ satisfies $\phi = \Smooth \! \phi$, then $(M^\phi_n(t))_{n \in \N_0}$ is a bounded martingale with respect to the filtration $(\F_n)_{n \in \N_0}$, where $\F_n = \sigma(\xi_u:\,|u| < n)$, $n \in \N_0$. Therefore, it converges almost surely and in mean to some $M^\phi_\infty(t)$, with $\E[M^\phi_\infty(t)]=\phi(t)$. From \eqref{eq:M-limit=1} we then conclude $M_\infty^{\phi}(t) \geq \1\{T > t\}$ almost surely. 
\end{proof} 

\begin{proof}[Proof of Proposition \ref{prop:minimal-fixed-point}.] 
Let $\phi \in \cM$ with $\phi = \Smooth \! \phi$. By applying Lemma \ref{lemma:M^phi_n}, we obtain 
\begin{equation*} 
\phi(t)\ =\ \E[M^\phi_\infty(t)]\ \geq\ \Prob(T > t). 
\end{equation*} 
For every continuity point $t>0$ of $\Fbar$, the right-hand side equals $\Fbar(t)$. Therefore, 
by left continuity of $\Fbar$, we conclude that $\phi(t) \geq \Fbar(t)$ for all $t \geq 0$. 
\end{proof} 

From Proposition \ref{prop:minimal-fixed-point}, we can almost immediately deduce the following corollary, which was also obtained by Komj\'athy \cite[Claim 2.3]{Komjathy:2016}. 
\begin{corollary} \label{cor:minimal fixed point} 
Assume \eqref{eq:critical}. 
If $\Fbar(t)<1$ for some $t > 0$, then $\Fbar(t) < 1$ for all $t > 0$, meaning that an explosive process can explode at any positive time with positive probability. In particular, the following holds for any $\eps > 0$: $\xi$ yields an explosive CMJ process if and only if the same holds for $\xi(\cdot \cap [0,\eps])$. 
\end{corollary} 
\begin{proof} 
If $\Fbar(t)<1$ for some $t$, then $\Fbar$ is a non-trivial fixed point of $\Smooth$. 
Suppose, for contradiction, that there exists $c>0$ such that $\Fbar(c)=1$. 
Define $\phi(t) = \Fbar(t+c)$ for $t \geq 0$. 
Then $\phi$ is an element of $\cM$ that satisfies $\Smooth \! \phi = \phi$, whence by Proposition \ref{prop:minimal-fixed-point} we have $\Fbar(t) \leq \phi(t) = \Fbar(t + c)$ for all $t \in \R$. 
But since $\Fbar$ is non-increasing we conclude that it must be constant, which contradicts $\Fbar(t)<1$ due to $\Fbar(0) = 1$. 
Therefore, $\Fbar(t) < 1$ for all $t > 0$, that is, the CMJ process can explode at any time. 
\end{proof} 

Before proving Theorem \ref{thm:uniqueness}, we establish two auxiliary results. 
\begin{lemma}\label{lemma:F<1} 
Assume that $\xi$ satisfies \eqref{eq:critical}. Then $\Fbar(t) > 0$ for all $t > 0$. 
\end{lemma} 
\begin{proof} 
Since $\xi$ is locally finite, continuity from above implies that $\xi(t) \downarrow \xi(0)$ almost surely as $t \downarrow 0$. In particular, $\Prob(\xi(t) = 0) \to \Prob(\xi(0)=0) > 0$, 
where the strict inequality follows from \eqref{eq:critical} and $\Prob(\xi(0)=1) < 1$. 
Consequently, there exists an $\eps > 0$ such that 
\begin{equation}\label{eq:min-lower-bound} 
\Prob(\xi(\eps) = 0) > 0. 
\end{equation} 
Now let $t>0$ be such that $\Fbar(t - \eps) > 0$ and define $M_1 \defeq \min_{|u|=1} S_u$. We then have 
\begin{align*} 
\Fbar(t)\ \geq\ \E\bigg[ \prod_{|u|=1} \Fbar(t - S_u) \1\{M_1 > \eps\}\Bigg]. 
\end{align*} 
Since \eqref{eq:min-lower-bound} implies $\Prob(M_1 > \eps) > 0$, 
and since almost surely only finitely many $|u|=1$ satisfy $S_u \leq t$ due to $\xi$ being locally finite, the integrand on the right-hand side is positive with positive probability. Hence the same is true for the left-hand side. 
\end{proof} 
\begin{lemma} \label{lemma:generation-explosion} 
Assume that $\xi$ satisfies \eqref{eq:critical} and yields an explosive CMJ process. 
Let $t > 0$ be a continuity point of $\Fbar$, and define 
\begin{equation} \label{eq:generation-explosion} 
Z_n(t) \defeq \sum_{|u| = n}\1\{S_u \leq t\}. 
\end{equation} 
Then we have $Z_n(t) \to \infty$ almost surely on $\{T \leq t\}$. 
\end{lemma} 

\begin{proof} 
Let $(M_n^{\Fbar}(t))_{n \in \N_0}$ be the sequence defined in Lemma \ref{lemma:M^phi_n} with $\phi = \Fbar$. 
Then, as $n\to\infty$, $M_n^{\Fbar}(t)$ converges almost surely and in mean to some $M_\infty^{\Fbar}(t)$, which satisfies $\E[M_\infty^{\Fbar}(t)] = \Fbar(t)$ and $M^{\Fbar}_\infty(t) = 1$ on $\{T > t\}$ almost surely. 
Thus we obtain 
\begin{align*} 
\Fbar(t)\ &=\ \E[M_\infty^{\Fbar}(t)\1\{T > t\}] + \E[M_\infty^{\Fbar}(t)\1\{T \leq t\}] \\ 
&=\ \Prob(T > t) + \E[M_\infty^{\Fbar}(t)\1\{T \leq t\}]. 
\end{align*} 
Now if $t > 0$ is a continuity point of $\Fbar$, then $\Fbar(t) = \Prob(T > t)$, and we conclude that 
\begin{equation*} 
M_\infty^{\Fbar}(t) = 0 \quad \text{almost surely on } \{T \leq t\}. 
\end{equation*} 
Note that the assumption that $\xi$ yields an explosive CMJ process implies $\Prob(T \leq t) > 0$ by Corollary \ref{cor:minimal fixed point}. 
Since $F(t) < 1$ by Lemma \ref{lemma:F<1}, there exists $c \in (0, \infty)$ such that $\log(1 - x) \geq -cx$ for all ${x \leq F(t)}$. 
Thus, we have 
\begin{align*} 
M_n^{\Fbar}(t)\ &=\ \exp\bigg( \sum_{|u|=n}\log(1 - F(t - S_u)) \bigg) 
\geq\ \exp\bigg(-c \sum_{|u|=n} F(t - S_u)\bigg). 
\end{align*} 
Since $M_n^{\Fbar}(t) \to 0$ on $\{T \leq t\}$, it follows that 
\begin{equation*} 
\lim_{n \to \infty}\sum_{|u|=n} F(t - S_u) = \infty \quad \text{almost surely on }\{T \leq t\}. 
\end{equation*} 
Finally, since $F(t - S_u) \leq \1\{S_u \leq t\}$, the claim follows. 
\end{proof} 

We are now ready to prove Theorem \ref{thm:uniqueness}. 

\begin{proof}[Proof of Theorem \ref{thm:uniqueness}] 
We start by deducing assertion (a) from (b). 
Given a non-trivial fixed point $\phi \in \cM$ of $\Smooth$, we set $c \coloneqq \sup\{t > 0\,:\, \phi(t) = 1\}$ and $\phi^c(t) \coloneqq \phi(t + c)$, $t \in \R$. 
Then we have $\phi^c \in \cM$ and $\phi^c(t) < 1$ for all $t > 0$, so that assertion (b) implies $\Smooth^n \! \phi^c \to \Fbar$ as $n \to \infty$. 
Since $\Smooth \! \phi^c = \phi^c$ it follows $\Fbar(t) = \phi^c(t) = \phi(t + c)$ for all $t \in \R$. 

It remains to prove (b). 
Let $t > 0$ and $(M_n^{\phi}(t))_{n \in \N_0}$ be as defined in Lemma \ref{lemma:M^phi_n}. 
We know that $M_n^\phi(t) = 1$ eventually as $n \to \infty$ almost surely on $\{T > t\}$. 
If the CMJ process is not explosive, then $T = \infty$ and hence $M_n^\phi(t) \to 1$ almost surely. 
Thus, we have 
$$ \Smooth^n \! \phi(t)\ =\ \E[M_n^\phi(t)]\ \to\ 1\ =\ \Fbar(t) $$ 
by the dominated convergence theorem. 

\vspace{.1cm} 
It remains to consider the case where the CMJ process is explosive. 
For all $\eps > 0$ such that $t-\eps >0$ is a continuity point of $\Fbar$, we find 
\begin{equation*} 
M_n^\phi(t)\ \leq\ \prod_{\substack{|u|=n\\S_u \leq t-\eps}} \phi(t - S_u)\ \leq\ \phi(\eps)^{Z_n(t - \eps)}, 
\end{equation*} 
which on $\{T \leq t - \eps\}$ converges to zero by Lemma \ref{lemma:generation-explosion} and the assumption $\phi(t) < 1$ for all $t > 0$. 
Thus, $M_n^\phi(t)\to 0$ almost surely on $\{T \leq t - \eps\}$. 
Letting $\eps \downarrow 0$, we conclude that $M_n^\phi(t)\to 0$ almost surely on $\{T < t\}$. 
Therefore, for every continuity point $t > 0$ of $\Fbar$, we have shown that 
\begin{equation*} 
\lim_{n \to \infty} M_n^\phi(t) = \1\{T \geq t\} \text{ almost surely.} 
\end{equation*} 
Since $\Smooth^n \! \phi(t) = \E[M_n^\phi(t)]$, applying the dominated convergence, we deduce that $\Smooth^n \! \phi(t)\to\Fbar(t)$ for all continuity points $t > 0$ of $\Fbar$. 

\vspace{.1cm} 
To extend this to all $t > 0$ we adapt an argument from \cite{Grishechkin:1987}. 
Let $t > 0$ be arbitrary, and let $\eps > 0$ be such that $t - \eps >0$ is a continuity point of $\Fbar$. Define 
\begin{equation*} 
\phi^\eps(x)\ \defeq\ \begin{cases} 
\phi(x+\eps), & x > 0, \\ 
1, & x \leq 0. 
\end{cases} 
\end{equation*} 
Then we have $(\Smooth^n \! \phi^\eps)(t - \eps) \to \Fbar(t - \eps)$ as $n \to \infty$. 
Moreover, we note that $\phi(x) \leq \phi^{\eps}(x - \eps) \eqqcolon (\theta_\eps \phi^{\eps})(x)$ for all $x \in \R$ where $\theta_\eps$ denotes the appropriate translation. 
Thus, using that $\Smooth$ is order-preserving and commutes with translation, for sufficiently large $n$ we have 
\begin{equation*} 
(\Smooth^n \! \phi)(t) \leq (\Smooth^n (\theta_\eps\phi^\eps))(t) = (\Smooth^n \! \phi^{\eps})(t - \eps) < \Fbar(t - \eps) + \eps. 
\end{equation*} 
Additionally, from \cite[Lemma 2.1]{Komjathy:2016}, we know that $\Smooth^n\! \1_{(-\infty,0]} \to \Fbar$ as $n \to \infty$. Hence, for~sufficiently large $n$, we have 
\begin{equation*} 
\Smooth^n \! \phi(t)\ \geq\ \Smooth^n \! \1_{(-\infty,0]}(t)\ >\ \Fbar(t) - \eps, 
\end{equation*} 
Combining these results, we obtain that 
\begin{equation*} 
\eps\ >\ \Fbar(t) - \Smooth^n \! \phi(t)\ >\ \Fbar(t) - \Fbar(t - \eps) - \eps 
\end{equation*} 
for all sufficiently large $n$. Finally, by the left-continuity of $\Fbar$, and by letting $\eps \downarrow 0$, we conclude that $\Smooth^n \! \phi(t) \to \Fbar(t)$. 
This~completes the proof of (b), and hence the proof of Theorem \ref{thm:uniqueness}. 
\end{proof} 

In the remainder of this section we collect two comparison results for later use. The following lemma, originally proved in \cite{Komjathy:2016} using an argument based on the smoothing transform, now follows almost immediately as a corollary of Theorem~\ref{thm:uniqueness}. 
Recall that for a point process $\xi$ and $t \geq 0$, we write $\xi(t) = \xi[0,t]$. 

\begin{lemma}[{\cite[Claim 2.7]{Komjathy:2016}}]\label{lemma:T-comparison} 
Let $\xi$ and $\xi'$ be point processes with corresponding smoothing transforms $\Smooth$ and $\Smooth'$, respectively. 
Suppose that there exists $t_0 > 0$ such that 
\begin{equation*} 
\Smooth \! \phi(t)\,\geq\,\Smooth' \! \phi(t)\quad\text{for all }\phi \in \cM\text{ and }t\leq t_0. 
\end{equation*} 
Then, if $\xi$ yields an explosive CMJ process, so does $\xi'$. 
In this case, the corresponding~explosion times $T$ and $T'$ satisfy 
\begin{equation}\label{eq:explosion-time-domination} 
\Prob(T \geq t)\,\geq\,\Prob(T' \geq t)\quad \text{for all }t \leq t_0. 
\end{equation} 
\end{lemma} 

\begin{proof} 
We claim that, for all $n \in \N$, 
\begin{equation} \label{eq:comparison claim} 
\Smooth^n \! \phi(t) \geq (\Smooth')^n \phi(t) \quad \text{ for all } t \leq t_0 \text{ and } \phi \in \cM. 
\end{equation} 
Notice that this implies the assertion of the lemma 
by choosing $\phi = \1_{(-\infty,0]}$ and then passing to the limit 
using Theorem \ref{thm:uniqueness}(b). 
For \eqref{eq:comparison claim}, the case $n=1$ holds by assumption. 
Further, for any $n \in \N$, for which \eqref{eq:comparison claim} is true, 
if $\phi \in \cM$ and $t \leq t_0$, then 
\begin{equation*} 
\Smooth^{n+1} \! \phi(t) 
= \Smooth (\Smooth^n \! \phi)(t) \geq \Smooth' (\Smooth^n \! \phi)(t) \geq \Smooth' ((\Smooth')^n \phi)(t), 
= (\Smooth')^{n+1} \phi(t), 
\end{equation*} 
where the first inequality is obtained from \eqref{eq:comparison claim} by taking $n=1$ 
and replacing $\phi$ with $\Smooth^n \! \phi$ and the second inequality follows from the induction hypothesis. 
By induction, \eqref{eq:comparison claim} holds for all $n \in \N$. 
\end{proof} 

A simple consequence of this result is the following proposition, which generalizes \cite[Thm. 4]{Grey:1974}. For completeness, we include a proof. 

\begin{proposition}[{\cite[Thm. 3.7]{Komjathy:2016}}] \label{prop:comparison} 
Let $\xi$ and $\xi'$ be point processes such that $\xi'$ dominates $\xi$ at zero, i.e., there exists a coupling $(\tilde{\xi}, \tilde{\xi}')$ of $\xi$ and $\xi'$ such that for some $t_0 > 0$ and all $t \leq t_0$, we have $\tilde{\xi}(t) \leq \tilde{\xi}'(t)$ almost surely. 
If $\xi$ yields an explosive CMJ process, then so does $\xi'$. 
In this case, the respective explosion times $T$ and $T'$ satisfy \eqref{eq:explosion-time-domination}. 
\end{proposition} 

\begin{proof} 
Let $f:\R \to \R$ be a bounded, left-continuous, increasing function with $f(t) = 0$ for all $t \leq 0$, and $\nu$, $\nu'$ 
be Borel measures on $[0,\infty)$ 
such that $\nu(t)\leq\nu'(t)$ for all $t \leq t_0$. 
Then there exists a Borel measure $\mu_f$ on $\R$ such that $\mu_f[a,b) = f(b) - f(a)$ for all $a<b$, and we obtain 
\begin{align*} 
f \ast \nu(t) &\defeq \int f(t - x) \nu(\d x) = (\mu_f \ast \nu)[0,t) \\ 
&= (\nu \ast \mu_f)[0,t)\ =\ \int \nu[0,t-x) \, \mu_f(\d x) \\ 
&\leq\ \int \nu'[0,t - x) \, \mu_f(\d x)\ =\ (\nu' \ast \mu_f)[0,t)\\ 
&= (\mu_f \ast \nu')[0,t) = f \ast \nu'(t) 
\end{align*} 
for all $t \leq t_0$. 
Conversely, if $f$ is decreasing, we obtain $f \ast \nu(t) \geq f \ast \nu'(t)$ for all $t \leq t_0$. 
Now let $(\tilde{\xi}, \tilde{\xi}')$ be a coupling of $\xi$ and $\xi'$ such that $\tilde{\xi}(t) \leq \tilde{\xi}'(t)$ almost surely, for some $t_0 > 0$ and all $t \leq t_0$. 
Then, for all $\phi \in \cM$ and $t \leq t_0$, we have 
\begin{align*} 
\Smooth \! \phi(t)\ &=\ \E\Big[\exp \Big(\int \log \phi(t - x) \, \xi(\dx) \Big)\Big]\ =\ \E\big[\exp\big(\log \phi \ast \tilde{\xi}(t)\big)\big] \\ 
&\geq\ \E\big[\exp\big(\log \phi \ast \tilde{\xi}'(t) \big) \big] = \Smooth' \! \phi(t). 
\end{align*} 
By Lemma \ref{lemma:T-comparison}, the result follows. 
\end{proof}
We provide a second proof of the proposition using a coupling argument.
\begin{proof}[Second proof of Proposition \ref{prop:comparison}.]
By Proposition \ref{prop:minimal-fixed-point} it suffices to prove \eqref{eq:explosion-time-domination}. 
From the Andersen-Jessen theorem we obtain a family $(\xi_u,\xi_u')_{u \in \mathcal{I}}$ of i.i.d.\ pairs on a common probability space such that $\xi_u$ has the same law as $\xi$, $\xi'_u$ has the same law as $\xi'$, and $\xi_u(t) \leq \xi_u'(t)$ almost surely for all $t \leq t_0$, for all $u \in \mathcal{I}$. 
Now we use these families to construct CMJ processes with reproduction $\xi$ and $\xi'$, respectively, as described in Section \ref{subsec:General branching process}.
Let $\cZ(t)$ and $\cZ'(t)$ denote the respective number of individuals that are born up to time $t$.
Then we have $\cZ(t) \leq \cZ'(t)$ almost surely for all $t \leq t_0$.
Therefore, with $T$ and $T'$ denoting the associated explosion times, it follows that $\{T < t \} \subseteq \{T' < t\}$ almost surely, which implies \eqref{eq:explosion-time-domination}.
\end{proof}
\section{Sufficient explosion criteria} 
\subsection{Comparison with a Galton-Watson process in varying environment -- Proof of Theorem \ref{Thm:sufficient-explosion-GWVE}} \label{subsec:sufficient-explosion-GWVE} 

The central idea behind the proof of Theorem \ref{Thm:sufficient-explosion-GWVE} is to construct a Galton-Watson process in varying environment that has the property that its survival implies the explosion of the original CMJ process. 
This idea has already been used in \cite{Jagers:1975} to construct an example of 
an explosive CMJ process satisfying \eqref{eq:critical} and \eqref{eq:finitemu_+}. 
To establish the survival of this process, we use the following result by Kersting \cite{Kersting:2020}: 

\begin{theorem}[{\cite[Theorem 1]{Kersting:2020}}]\label{thm:kersting} 
Let $(Y_n)_{n \in \N_0}$ be a sequence of $\N_0$-valued random variables and 
$(Z_n)_{n \in \N_0}$ be a Galton-Watson process in varying environment $(Y_n)_{n \in \N_0}$, i.e., the offspring distribution of an individual in the $n$-th generation is distributed like $Y_n$ for each $n \in \N_0$. 
Suppose that there exists $c \in (0,\infty)$ such that for all $n \in \N$ 
\begin{equation}\label{eq:Kersting-A} 
\E[Y_n^2 \1\{Y_n \geq 2\}]\ \leq\ c\,\E[Y_n\1\{Y_n \geq 2\}]\cdot\E[Y_n\,|\,Y_n\geq 1]. 
\end{equation} 
Then $(Z_n)_{n \in \N_0}$ survives with positive probability if and only if 
\begin{equation*} 
\sum_{n \in \N} \frac{\nu_n}{\mathsf{m}_{n-1}} < \infty\quad\text{and}\quad\lim_{n \to \infty} \mathsf{m}_n \in (0,\infty] \text{ exists,} 
\end{equation*} 
where $\mathsf{m}_n \defeq \E[Z_n] = \E[Y_1] \cdot \ldots \cdot \E[Y_n]$ and $\nu_n \defeq \E[Y_n(Y_n-1)]/\E[Y_n]^2$ for $n \in \N_0$. 
\end{theorem} 

Theorem \ref{Thm:sufficient-explosion-GWVE} follows almost directly from the next lemma. 
\begin{lemma} \label{Lem:explosion via GWPVE} 
Assume \eqref{eq:critical}, \eqref{eq:finitemu_+} and \eqref{eq:positivemu_+}, and that there exists $\eps > 0$ such that $\E[\xi(\varepsilon)^2]<\infty$. 
Further, suppose that there exists a decreasing, summable sequence $(a_j)_{j \in \N_0}$ of positive numbers satisfying 
\begin{align} \label{eq:explosion via GWPVE} 
\sum_{n=1}^\infty \prod_{j=1}^n \frac1{\mu(a_j)}\,<\,\infty. 
\end{align} 
Then $\xi$ yields an explosive CMJ process. 
Moreover, each of the following conditions is sufficient for \eqref{eq:explosion via GWPVE}: 
\begin{align} 
&\sum_{n=1}^\infty \exp\bigg(-\sum_{j=1}^n \mu_+(a_j) \bigg) \exp\bigg(\frac12\sum_{j=1}^n \mu_+(a_j)^2 \bigg)\,<\,\infty \label{eq:GWVE-sufficient-i}, \\ 
&\sum_{n=1}^\infty \exp\bigg(-\sum_{j=1}^n \mu_+(a_j) \bigg)\,<\,\infty\quad\text{and}\quad\sum_{j=1}^\infty \mu_+(a_j)^2\,<\,\infty \label{eq:GWVE-sufficient-ii}, \\ 
&\sum_{n=1}^\infty \exp\bigg(-\delta\sum_{j=1}^n \mu_+(a_j) \bigg)\,<\,\infty \text{ for some }\delta \in (0,1) \label{eq:GWVE-sufficient-iii}. 
\end{align} 
\end{lemma} 

\begin{proof} 
Without loss of generality, we assume that $a_j \leq \varepsilon$ for every $j \in \N_0$. 

For $u \in \I$ with $|u|=j \in \N_0$, define $\eta_u \defeq \xi_u[0,a_j]$. 
Let $Y_j$ be a random variable with the same distribution as $\xi[0,a_j]$, for $j \in \N_0$. 
Consider the Galton-Watson process $(Z_n)_{n \in \N_0}$ in varying environment $(Y_n)_{n \in \N_0}$, where the number of offspring of individual $u$ is $\eta_u$. 
Then, using $(\cZ_t)_{t \geq 0}$ from \eqref{eq:Z-def} and $a = \sum_{j=1}^\infty a_j$, we clearly have $\cZ_a \geq \sum_{n=0}^\infty Z_n$. Therefore, on the set where $(Z_n)_{n \in \N_0}$ survives, the process $(\cZ_t)_{t \geq 0}$ explodes. 
Thus, the problem of explosion has been reduced to the problem of survival of a Galton-Watson process 
in varying environment. 
We apply Theorem \ref{thm:kersting} to establish survival. 
To begin, we verify that condition \eqref{eq:Kersting-A} holds and note first that 
$$ \E[Y_n^2 \1{\{Y_n \geq 2\}}]\ \leq\ \E[Y_n^2]\ \leq\ \E[\xi(\varepsilon)^2]\ <\ \infty. $$ 
Next, we have: 
\begin{equation} \label{eq:Kersting (A)} 
\E[Y_n \1{\{Y_n \geq 2\}}]\cdot\E[Y_n\, |\, Y_n \geq 1] 
\ =\ \E[Y_n \1{\{Y_n \geq 2\}}] \frac{\E[Y_n]}{\Prob(Y_n \geq 1)} \ge \E[Y_n \1{\{Y_n \geq 2\}}]. 
\end{equation} 
It remains to show that $\E[Y_n \1{\{Y_n \geq 2\}}]$ is bounded away from $0$ as $n \to \infty$. 
Indeed: 
\begin{align*} 
\E[Y_n \1{\{Y_n \geq 2\}}]\ =\ \E[Y_n] - \Prob(Y_n=1)\ =\ \mu(a_n) - \Prob(\xi(a_n) =1)\ \to\ 1-\Prob(\xi(0)=1)\ >\ 0 
\end{align*} 
by \eqref{eq:critical}. Thus, Theorem \ref{thm:kersting} applies. 
We note that $\E[Y_n] = \mu(a_n) \to 1$ as $n \to \infty$ and that 
\begin{align*} 
\nu_n\ =\ \frac{\E[Y_n(Y_n-1)]}{\E[Y_n]^2}\ =\ \frac{\E[Y_n^2]}{\E[Y_n]^2} - \frac{1}{\E[Y_n]}\ \leq\ \E[\xi(\varepsilon)^2]\ <\ \infty\quad\text{for each }n\in\N. 
\end{align*} 
Consequently, $(Z_n)_{n \in \N_0}$ survives with positive probability if 
\begin{equation*} 
\sum_{n=1}^\infty \frac1{\E[Y_1]\cdots\E[Y_n]}\ <\ \infty. 
\end{equation*} 
The claim now follows from \eqref{eq:explosion via GWPVE} and 
\begin{equation*} 
\sum_{n=1}^\infty \frac1{\E[Y_1]\cdots\E[Y_n]}\ =\ \sum_{n=1}^\infty \prod_{j=1}^n \frac1{\mu(a_j)}. 
\end{equation*} 
Finally, to show the sufficiency of \eqref{eq:GWVE-sufficient-i}, \eqref{eq:GWVE-sufficient-ii}, or \eqref{eq:GWVE-sufficient-iii} for \eqref{eq:explosion via GWPVE}, note that 
\begin{equation*} 
\prod_{j=1}^n \frac1{\mu(a_j)}\ =\ \exp\bigg(-\sum_{j=1}^n \log(1+\mu_+(a_j)) \bigg). 
\end{equation*} 
Thus \eqref{eq:GWVE-sufficient-i} follows from the inequality $\log(1+x) \geq x-\frac12 x^2$ for $x \geq 0$, \eqref{eq:GWVE-sufficient-ii} follows from \eqref{eq:GWVE-sufficient-i}, and \eqref{eq:GWVE-sufficient-iii} follows from the inequality $\log(1+x) \geq \delta x$ for sufficiently small $x \geq 0$ (where in the last case $\varepsilon > 0$ might have to be adjusted to make sure that $x = \mu_+(a_j)$ is sufficiently small). 
\end{proof} 
\begin{proof}[Proof of Theorem \ref{Thm:sufficient-explosion-GWVE}.] 
Choose $a_j \defeq 1/(j \log^{1+r} j)$ for some $0<r<\delta$ and apply the sufficient condition \eqref{eq:GWVE-sufficient-iii}. 
\end{proof} 
\subsection{Comparison with a Bellman-Harris process - Proof of Theorem \ref{Thm:sufficient-explosion-Amini}} \label{subsec:sufficient-explosion-Amini} 

For a random variable $X$, let $F_X = \Prob(X \leq \cdot)$ denote its distribution function and define the generalized inverse of $F_X$ as 
\begin{equation*} 
F_X^{-1}(y) = \inf\{x\,:\,F_X(x) \geq y\}. 
\end{equation*} 
The CMJ process with reproduction point process $\xi = Z\delta_W$ for independent $Z$ and $W$ is known in the literature as \emph{Bellman-Harris} or \emph{age-dependent} branching process. 
The proof of Theorem \ref{Thm:sufficient-explosion-Amini} builds on the characterization of explosion of Bellman-Harris processes with a certain type of heavy-tailed offspring distribution due to Amini et al. 
We will use the following reformulation of \cite[Thm 1.3]{Amini:2013} by Komj\'athy. 

\begin{theorem}[{\cite[Lemma 5.8]{Komjathy:2016}}] \label{thm:Komjathy 5.8} 
Let $Z$ and $W$ be independent random variables taking values in $\N_0$ and $[0,\infty)$, respectively. Suppose there exists $\delta > 0$ such that $F_Z$ satisfies 
\begin{equation}\label{eq:plump-power} 
\frac{1}{t^{1-\delta}}\ \leq\ 1-F_Z(t)\ \leq\ \frac{1}{t^\delta}\quad\text{for all sufficiently large } t > 0. 
\end{equation} 
Then the point process $\xi = Z\delta_W$ 
yields an explosive CMJ process if and only if there exists $C > 0$ such that 
\begin{equation*} 
\int_C^\infty F_W^{-1}\big(\e^{-y} \big) \frac{\d y}{y}\ <\ \infty, 
\end{equation*} 
or equivalently, if there exists $\eps \in (0,1)$ such that 
\begin{equation}\label{eq:Komjathy} 
\int_0^\eps F_W^{-1}(x) \frac{\d x}{x |\log x|}\ <\ \infty. 
\end{equation} 
\end{theorem} 

In the following, we extend the applicability of Theorem~\ref{thm:Komjathy 5.8} by relating the explosion behavior of certain CMJ processes to that of Bellman-Harris processes. 
\begin{lemma}\label{lemma:tree-coupling} 
Let $Z$ and $W$ be independent random variables with values in $\N_0$ and $(0,\infty)$, respectively. Let $W_n$, $n \in \N$ be i.i.d. copies of $W$. Define the point processes $\xi \coloneqq Z \delta_W$ and $\xi' \coloneqq \sum_{j=1}^Z \delta_{W_j}$. Then $\xi$ yields an explosive CMJ process if and only if $\xi'$ does. 
\end{lemma} 

\begin{proof} 
We construct a coupling of the two CMJ processes. 
Recall that a CMJ process can be constructed from an i.i.d.\ family of point processes $(\xi_u)_{u \in \I}$ 
where $\I$ is the set of labels for all potential individuals, see Section \ref{subsec:General branching process}. 
It thus suffices to define the point processes $(\xi_u)_{u \in \I}$ and $(\xi_u')_{u \in \I}$ on the same probability space such that $(\xi_u)_{u \in \I}$ and $(\xi_u')_{u \in \I}$ are i.i.d. families of copies of $\xi$ and $\xi'$, respectively, and the corresponding CMJ processes exhibit the same explosion behavior. 

Let $(Z_u)_{u \in \I}$ and $(W_u)_{u \in \I}$ be independent families of i.i.d.\ copies of $Z$ 
and $W$, respectively. 
For $u \in \I$, define 
\begin{align*} 
\xi_u \coloneqq Z_u \delta_{W_u}\qquad\text{and}\qquad 
\xi'_u \coloneqq \sum_{j=1}^{Z_u} \delta_{W_{uj}}. 
\end{align*} 
Then $(\xi_u)_{u \in \I}$ and $(\xi'_u)_{u \in \I}$ have the asserted distributions. 
Moreover, if $(S(u))_{u \in \I}$ and $(S'(u))_{u \in \I}$ denote the birth times in the respective CMJ processes, 
then an induction using \eqref{eq:S-def} yields 
\begin{equation*} 
S(uj) = W_\varnothing + S'(u),\quad u \in \I,\; j \in \N. 
\end{equation*} 
As a consequence, the respective minima $M_n \coloneqq \min_{|u|=n} S(u)$ 
and $M_n' \coloneqq \min_{|u|=n} S'(u)$ satisfy 
\begin{equation*} 
M_n = W_{\varnothing} + M_{n-1}',\quad n \in \N. 
\end{equation*} 
Therefore, in view of \eqref{eq:T-approx}, the respective explosion times $T$ and $T'$ are linked by $T = W_{\varnothing} + T'$ and the claim follows. 
\end{proof} 


Another comparison can be made with the process 
\begin{equation}\label{eq:xi''} 
\xi'' = Y\delta_0 + \delta_W 
\end{equation} 
where $Y$ and $W$ are independent, $\E[Y] = 1$, and $\Prob(Y = 1) < 1$. 
In the associated CMJ process the progeny of each individual generates an instantaneous critical Galton-Watson process at its birth time, with offspring distribution $Y$, which has almost surely finite total population $Y_\infty$, say. 
Each individual in this Galton-Watson process has exactly one sibling 
that receives an independent, positive displacement distributed like $W$. 
Thus, instead of $\xi''$, we might consider the reproduction point process $\xi = \sum_{j=1}^{Y_\infty} \delta_{W_j}$, where the $W_j$ for $j\in\N$ are i.i.d.\ with the same distribution as $W$, independent of $Y_\infty$. 
Using Theorem \ref{thm:Komjathy 5.8} and Lemma \ref{lemma:tree-coupling}, we can characterize the explosion of the CMJ processes associated with $\xi''$. 
\begin{lemma}\label{lemma:Ydelta_0+delta_W} 
Let $Y$ and $W$ be independent random variables taking values in $\N_0$ and $(0,\infty)$, respectively, such that $\E[Y] = 1$ and $\Prob(Y = 1) < 1$. 
Suppose further that the distribution of $Y$ belongs to the domain of attraction of the normal distribution or, more generally, a stable distribution with index $\alpha \in (1,2]$. 
Then the point process 
\begin{equation*} 
\xi\,=\,Y\delta_0 + \delta_W 
\end{equation*} 
yields an explosive CMJ process if and only if \eqref{eq:Komjathy} holds. 
\end{lemma} 
\begin{proof} 
As described above, we can replace $\xi$ by the process $\xi'' = \sum_{j=1}^{Y_\infty} \delta_{W_j}$, where $Y_\infty$ is the total population of a critical Galton-Watson tree with offspring distribution $Y$, taken to be independent of $(W_j)_{j \in \N}$. 
We now provide an alternative argument based on the fixed-point equation 
for the fact that $\xi$ yields an explosive CMJ process if and only if $\xi''$ does. 
For $\xi'' = \sum_{j=1}^{Y_\infty} \delta_{W_j}$, the equation $\phi = \Smooth\! \phi$ for $\phi \in \cM$ takes the form 
\begin{equation}\label{eq:fixedpoint-xi''} 
\phi(t)\,=\,h(\E[\phi(t - W)]),\quad t \geq 0 
\end{equation} 
with $h(x) \defeq \E[x^{Y_\infty}]$. 
For $\xi = Y \delta_0 + \delta_W$, the corresponding equation becomes 
\begin{equation}\label{eq:fixedpoint-xi} 
\phi(t)\,=\, f(\phi(t)) \E[\phi(t-W)],\quad t \geq 0 
\end{equation} 
where $f(x) \defeq \E[x^Y]$. 
Classical theory for critical Galton-Watson processes gives the relationship 
\begin{equation*} 
h^{-1}(x)\,=\,\frac{x}{f(x)},\quad x \in [0,1] 
\end{equation*} 
for the inverse function $h^{-1}$ of $h$, 
which shows that the equations \eqref{eq:fixedpoint-xi''} and \eqref{eq:fixedpoint-xi} are equivalent. 
Since $\Fbar$ is the minimal solution of this equation, $\xi$ yields an explosive CMJ process if and only if $\xi''$ does. 

By Theorem \ref{thm:Komjathy 5.8} and Lemma \ref{lemma:tree-coupling}, the condition \eqref{eq:Komjathy} characterizes explosion of the CMJ process with reproduction point process $\xi''$ if we show that the distribution of $Y_\infty$ satisfies \eqref{eq:plump-power}. 
We apply a result by Dwass (\cite{Dwass:1969}, see also \cite[p.~104f]{Kolchin:1986} for a simple proof), which states that 
\begin{equation*} 
\Prob(Y_\infty = n) = \frac1{n}\Prob(S_n = n-1) \quad \text{ for all } n \in \N 
\end{equation*} 
where $(S_n)_{n \in \N_0}$ has i.i.d.~increments with $S_0=0$ and $S_1 \sim Y$. 
Since $Y$ belongs to the domain of attraction of a stable law with index $\alpha \in (1,2]$, there exists a sequence $(a_n)_{n \in \N}$ of the form 
\begin{equation*} 
a_n = n^{1/\alpha} \ell(n) \quad \text{ as }n \to \infty 
\end{equation*} 
for some slowly varying function $\ell$ such that $a_n^{-1}(S_n - n)$ converges in distribution to the corresponding stable law as $n \to \infty$, see e.g.\ \cite[IX.8]{Feller:1971}. 
Using Stones's local limit theorem \cite[Thm.\ 1]{Stone:1967}, we obtain 
\begin{equation*} 
\Prob(S_n = n-1) \sim \frac{c}{a_n} \quad \text{ as }n \to \infty 
\end{equation*} 
for some constant $c \in (0,\infty)$. 
Summation gives that the tail of $Y_\infty$ is regularly varying with index $-1/\alpha$. 
In particular, the distribution of $Y_\infty$ satisfies \eqref{eq:plump-power}, due to $\alpha \in (1,2]$. 
Hence, condition \eqref{eq:Komjathy} characterizes explosion. 
\end{proof} 
\begin{proof}[Proof of Theorem \ref{Thm:sufficient-explosion-Amini}] 
Let $\xi$ be a point process $\xi$ as stated in the theorem, and define 
\begin{equation*} 
W\,\defeq\,\inf\{x > 0\,:\, \xi(0,x] > 0\}. 
\end{equation*} 
Observing that $\xi$ dominates $\xi' = \xi(0)\delta_0 + \delta_W$ at zero, it follows from Proposition \ref{prop:comparison} 
that, if $\xi'$ yields an explosive CMJ process, then so does $\xi$. 
The distribution of $W$ is given by 
\begin{equation*} 
F_W(t)\,=\,\Prob(\xi(0,t] > 0),\quad t \geq 0. 
\end{equation*} 
Next, let $t_0, \delta > 0$ be such that $\E[\xi(0,t_0]^{1+\delta}]< \infty$. Applying H\"older's inequality to $\E[\xi(0,t]]$ with $p=1+\delta$ and $q = (1+\delta)/\delta$, we obtain 
\begin{equation*} 
\Prob(\xi(0,t] > 0)\ \geq\ \bigg(\frac{\E[\xi(0,t]]}{\E[\xi(0,t]^{1+\delta}]^{1/(1+\delta)}}\bigg)^q\ \geq\ \frac{\mu_+(t)^q}{\E[\xi(0,t_0]^{1+\delta}]^{1/\delta}}\ \eqdef\ \frac{\mu_+(t)^q}{c} 
\end{equation*} 
for all $t \leq t_0$. 
For sufficiently small $y$ we now set $t \defeq \mu_+^{-1}((cy)^{1/q})$, which gives 
\begin{equation*} 
F_W(t) \geq \frac1c \mu_+\big(\mu_+^{-1}\big((cy)^{1/q}\big)\big)^q \geq y. 
\end{equation*} 
Now, since $F_W^{-1}(y) \leq t$ holds if and only if $y \leq F_W(t) = \Prob(\xi(0,t] > 0)$, we infer that 
\begin{equation*} 
F_W^{-1}(y)\ \leq\ \mu_+^{-1} \big( (cy)^{1/q} \big) 
\end{equation*} 
for some constant $c \in (0,\infty)$ and all sufficiently small $y$. 
Finally, by combining this with Lemma \ref{lemma:Ydelta_0+delta_W}, the assertion follows. 
\end{proof} 

\section{Explosion of CMJ processes with Poisson reproduction} 

\subsection{Proof of Theorem \ref{Thm:explosion-Poisson}} \label{subsec:explosion-poisson} 
Let $\xi$ be a Poisson point process satisfying \eqref{eq:critical}, \eqref{eq:finitemu_+} and \eqref{eq:positivemu_+}. 
Additionally, we impose that the cumulative distribution function $\mu_+$ is convex in a neighborhood of zero, and note that $\mu_+$ is continuous in that neighborhood. 
Using the Laplace functional 
\begin{equation*} 
\cL_\xi(u)\ \defeq\ \E\bigg[\exp\bigg( -\int u(x) \, \xi(\d x) \bigg)\bigg] 
\end{equation*} 
for measurable $u:[0,\infty) \to [0,\infty)$, we can express the smoothing transform \eqref{eq:ST} as 
\begin{equation*} 
\Smooth \! \phi(t) = \cL_\xi(-\log \phi(t - \cdot)),\quad t \geq 0,\ \phi \in \cM. 
\end{equation*} 
If $\xi$ is a Poisson point process with intensity measure $\mu$, the Laplace functional takes the form 
\begin{equation*} 
\cL_\xi(u)\ =\ \exp\Big( -\int \big(1 - \e^{-u(x)} \big) \, \mu(\d x) \Big), 
\end{equation*} 
so the smoothing transform becomes 
\begin{equation}\label{eq:poisson-smoothing} 
\Smooth \! \phi(t)\ =\ \exp\Big(- \int \big(1 - \phi(t - x) \big) \, \mu(\d x) \Big)\ =\ \exp\big(-(1-\phi) \ast \mu(t) \big),\quad t \geq 0. 
\end{equation} 
In particular, the distribution function $F$ of the explosion time satisfies 
\begin{equation*}
F\,=\,1 - \e^{-F \ast \mu}. 
\end{equation*} 
The proof of Theorem \ref{Thm:explosion-Poisson} follows the steps outlined in \cite{Grishechkin:1987}: 
First, we approximate the distribution function $F$ by a suitable piecewise constant function $\psi$ (see Lemma \ref{lemma:psi-comparison}). 
This approximation allows us to bound the integral in \eqref{eq:explosion-integral-cond} for some $\eps$ depending on $F$ (see Lemma \ref{lemma:int-bound}). 
Finally, we approximate $\mu_+$ from above by functions $\mu_n$, for which \eqref{eq:explosion-integral-cond} holds, and using the previously obtained bound, we conclude that the condition also holds for $\mu_+$. 
\begin{lemma}\label{lemma:psi-comparison} 
Fix $\delta \in (0,1)$. Suppose that $\mu_+:[0,\infty)\to [0,\infty)$ is a non-constant convex function that satisfies $\mu_+(0) = 0$ and the condition 
\begin{equation}\label{eq:sum-condition} 
\sum_{n \in \N} \frac{\mu_+^{-1}(\delta^n)}{n}\ <\ \infty. 
\end{equation} 
Define $a_0 \defeq \infty$, $a_n \defeq \sum_{k \geq n} \mu_+^{-1}(\delta^k)/k$ for $n \in \N$, and the function $\psi$ by 
\begin{equation*} 
\psi(t)\ \defeq\ \sum_{n \in \N_0} \delta^n \1_{(a_{n+1},a_n]}(t),\quad t \geq 0. 
\end{equation*} 
Then 
$$ F(t)\,\leq\,\psi(4t)\quad\text{for all }t\,\leq\,t_0\,\defeq \frac{\delta}{4(2 + \delta)}\,\mu_+^{-1}(1). $$ 
\end{lemma} 

\begin{proof} 
Since $\mu_+$ is convex and non-constant, it is continuous and strictly increasing and hence 
is the cumulative mass function of a continuous measure on $[0,\infty)$. 
Moreover, $\mu_+$ is a bijection of $[0,\infty)$, 
and its generalized inverse $\mu_+^{-1}$ is simply the inverse function. 
Let $\psi_0 = \psi$, and for $n \in \N$, define $\psi_{n}$ recursively by 
\begin{align*} 
\psi_{n+1}(t)\ &\defeq\ 1 - \exp\bigg( - \psi_{n}(t) - \int_0^{t/4} \psi_{n}(t - 4x) \, \mu_+(\d x) \bigg), 
\quad t \geq 0. 
\end{align*} 
Denote by $\Smooth$ the smoothing transform associated with the Poisson point process with intensity measure $\mu=\delta_0+\mu_+$ as given by \eqref{eq:ST} 
and recall its representation \eqref{eq:poisson-smoothing}. By induction, we observe that 
\begin{equation*} 
1 - \psi_n(4t)\ =\ \Smooth^n\big[1 - \psi(4(\cdot)) \big](t), 
\quad t \geq 0. 
\end{equation*} 
In particular, $1 - \psi_n(4(\cdot)) \in \cM$ due to $\Smooth$ being a self-map on $\cM$. 
By Theorem \ref{thm:uniqueness}(b), we infer that $\psi_n(4t) \to F(t)$ as $n \to \infty$ for all $t > 0$. 

\vspace{.1cm} 
To complete the proof, it remains to show that $\psi_n(t) \leq \psi(t)$ holds for all $t \leq 4 t_0$ and $n \in \N_0$. 
We will prove the slightly stronger statement $\psi_n(t) \leq \psi(t)$ for all $t \leq a_{k_0}$ via induction on $n$, where 
$k_0 \in \N_0$ is chosen such that 
\begin{equation}\label{eq:k_0} 
a_{k_0+1}\,<\,4 t_0\,\leq\,a_{k_0}. 
\end{equation} 
For $n = 0$, the claim is trivially true (base case). 
For the inductive step, assume $\psi_{n}(t)\le\psi(t)$ for all $t \leq a_{k_0}$ and some $n \in \N_0$. We will show that $\psi_{n+1}(t) \leq \psi(t)$ for all $t \in (a_{k+1},a_k]$ and $k\geq k_0$. 
If $k_0 = 0$, for $t > a_1$ we clearly have $\psi(t) = 1 \geq \psi_{n+1}(t)$. 
Thus we can assume $k \geq k_0 \vee 1$. 
Let $t \in (a_{k+1},a_k]$ for $k \geq k_0 \vee 1$. 
Then we have 
\begin{align*} 
\psi_{n+1}(t)\ \leq\ \psi_{n+1}(a_{k}) 
\ &=\ 1 - \exp\bigg(-\psi_n(a_{k}) - \int_0^{a_k/4} \! \psi_n(a_k-4x) \, \mu_+(\d x) \bigg) \\ 
&\leq\ 1 - \exp\bigg(-\psi(a_{k}) - \int_0^{a_k/4} \! \psi(a_k-4x) \, \mu_+(\d x) \bigg) \\ 
&=\ 1 - \exp\big(-(\psi(t) + I_k)\big) 
\end{align*} 
where $I_k$ represents the integral in the exponent. 
To estimate $I_{k}$, we first note that 
\begin{align*} 
I_k\ &=\ \int_0^{a_k/4} \! \psi(a_k-4x) \, \mu_+(\d x) 
\ =\ \sum_{j \geq k} \int_{[a_k-a_j, a_k - a_{j+1})/4} \! \psi(a_k-4x) \, \mu_+(\d x) \\ 
&=\ \sum_{j \geq k} \delta^j \bigg[\mu_+ \Big(\frac{a_k - a_{j+1}}4 \Big) - \mu_+\Big(\frac{a_k - a_{j}}4 \Big) \bigg] \\ 
&=\ \delta^k \sum_{j=0}^\infty \delta^j \bigg[\mu_+ \Big(\frac{a_k - a_{k+j+1}}4 \Big) - \mu_+\Big(\frac{a_k - a_{k+j}}4 \Big) \bigg] \\ 
&=\ \psi(t) \sum_{j=0}^\infty \delta^j (1 - \delta) \mu_+\Big(\frac{a_k - a_{k+j+1}}4 \Big), 
\end{align*} 
which we can express as 
\begin{equation*} 
\psi(t)\,\E\bigg[\mu_+\bigg(\frac14 \sum_{l=0}^{X} \frac{\mu_+^{-1}(\delta^{k+l})}{k+l} \bigg)\bigg] 
\end{equation*} 
where $X$ is a geometric random variable, namely, $\Prob(X = j) = (1 - \delta)\delta^j$, $j \in \N_0$. 
Using the convexity of $\mu_+$ and $\mu_+(0) = 0$, we have $\mu_+(x/4) \leq \mu_+(x)/4$ for all $x \geq 0$, leading to the bound 
\begin{equation}\label{eq:I_k-bound} 
I_k\ \leq\ \frac14 \psi(t)\,\E\bigg[\mu_+\bigg(\sum_{l=0}^{X} \frac{\mu_+^{-1}(\delta^{k+l})}{k+l} \bigg)\bigg]. 
\end{equation} 
To complete the proof, we need to show that the expectation on the right-hand side is bounded by $2 \delta^k = 2 \psi(t)$ for all $k \geq k_0 \vee 1$. 
Once this is established, we arrive at the inequality 
\begin{equation*} 
\psi_{n+1}(t)\ \leq\ 1 - \exp\Big(-\psi(t) \Big(1 + \frac{\psi(t)}{2} \Big) \Big). 
\end{equation*} 
By the elementary inequality 
\begin{equation*} 
1 - \exp\Big[-x\Big(1 + \frac{x}2 \Big) \Big] \leq x \quad \text{for all }x \in [0,1], 
\end{equation*} 
we obtain $\psi_{n+1}(t) \leq \psi(t)$, completing the induction. Regarding the elementary inequality, we note that 
the function on the left-hand side vanishes at the origin, has derivative $1$ there, 
and is increasing and concave on the positive halfline. 

\vspace{.1cm} 
Now, we bound the expectation in \eqref{eq:I_k-bound}. Let 
\begin{equation*} 
Y_k\ \defeq\ \sum_{l=0}^X \frac{1}{k+l}, \quad k \in \N 
\end{equation*} 
and consider the two possible outcomes $Y_k \leq 1$ and $Y_k > 1$ separately. 
We claim that for all $j \in \N_0$, 
\begin{equation}\label{eq:mu_+(sum)-bound} 
\mu_+\Bigg(\sum_{l=0}^{j} \frac{\mu_+^{-1}(\delta^{k+l})}{k+l}\Bigg)\ \leq\ 1. 
\end{equation} 
If this claim holds, then using ${Y_k \leq (X+1)/k}$, we obtain 
\begin{equation*} 
\E\bigg[\mu_+ \bigg(\sum_{l=0}^{X} \frac{\mu_+^{-1}(\delta^{k+l})}{k+l} \bigg) \1\{Y_k > 1\}\bigg]\ \leq\ \Prob(Y_k > 1)\ \leq\ \Prob(X > k-1)\ =\ \delta^{k}. 
\end{equation*} 
To prove \eqref{eq:mu_+(sum)-bound}, we first estimate 
\begin{equation*} 
\sum_{l=0}^{j} \frac{\mu_+^{-1}(\delta^{k+l})}{k+l}\ \leq\ a_k. 
\end{equation*} 
If $k_0 = 0$, then by \eqref{eq:k_0}, 
\begin{equation*} 
a_k\ \leq\ a_1 < 4t_0\ =\ \frac{\delta}{\delta + 2}\mu_+^{-1}(1)\ \leq\ \mu_+^{-1}(1), 
\end{equation*} 
which implies \eqref{eq:mu_+(sum)-bound}. 
If $k_0 > 0$, we proceed as follows. 
For $x_j \defeq \mu_+^{-1}(\delta^j)/j$, $j \in \N$, we have 
\begin{equation*} 
\frac{x_{j}}{x_{j+1}}\ =\ \frac{j+1}{j}\frac{\mu_+^{-1}(\delta^j)}{\mu_+^{-1}(\delta^{j+1})}\ \leq\ \frac2\delta, 
\end{equation*} 
where we use the concavity of $\mu_+^{-1}$ and that $\mu_+(0) = 0$, which imply $\mu_+^{-1}(\delta x) \geq \delta \mu_+^{-1}(x)$ for all $x \geq 0$. 
Now, using \eqref{eq:k_0}, we estimate 
\begin{equation*} 
a_{k_0}\ =\ a_{k_0 + 1} + x_{k_0}\ \leq\ a_{k_0 + 1}\Big(1 + \frac{x_{k_0}}{x_{k_0 + 1}}\Big) 
\leq 4t_0\Big(1 + \frac2\delta \Big)\ =\ \mu_+^{-1}(1). 
\end{equation*} 
Thus, with $a_k \leq a_{k_0}$, we again conclude \eqref{eq:mu_+(sum)-bound}. 

\vspace{.1cm} 
Finally, we estimate the expectation in \eqref{eq:I_k-bound} on $\{Y_k \leq 1\}$. 
Since $\mu_+$ is increasing, we have 
\begin{equation*} 
\mu_+ \Bigg(\sum_{l=0}^{X} \frac{\mu_+^{-1}(\delta^{k+l})}{k+l} \Bigg)\ \leq\ \mu_+ \Bigg( \frac{1}{Y_k} \sum_{l=0}^{X} \frac{\mu_+^{-1}(\delta^{k+l})}{k+l} \Bigg). 
\end{equation*} 
By Jensen's inequality, we see that 
\begin{equation*} 
\frac1{Y_k} \sum_{l=0}^{X} \frac{\delta^{k+l}}{k+l}\ \leq\ \delta^k. 
\end{equation*} 
Therefore, we obtain 
\begin{equation*} 
\E\Bigg[\mu_+ \Bigg(\sum_{l=0}^{X} \frac{\mu_+^{-1}(\delta^{k+l})}{k+l}\Bigg)\1\{Y_k \leq 1\}\Bigg]\ \leq\ \delta^k \Prob(Y_k \leq 1) \leq \delta^k. 
\end{equation*} 
Combining the two cases, we arrive at the final bound 
\begin{equation*} 
\E\Bigg[\mu_+ \Bigg(\sum_{l=0}^{X} \frac{\mu_+^{-1}(\delta^{k+l})}{k+l} \Bigg)\Bigg]\ \leq\ 2 \delta^k 
\end{equation*} 
for all $k\geq k_0\vee 1$. This completes the proof. 
\end{proof} 
\begin{lemma}\label{lemma:int-bound} 
Let $\mu_+:[0,\infty)\to [0,\infty)$ be a convex function that satisfies $\mu_+(0) = 0$, $\mu_+(t) > 0$ for all $t > 0$, and condition \eqref{eq:sum-condition} for some ${\delta \in (0,1)}$. 
Then, with $t_0$ from Lemma \ref{lemma:psi-comparison}, 
\begin{equation*} 
\int_0^{\delta F(t_0)} \frac{\mu_+^{-1}(x)}{x|\log x|} \, \d x\ \leq\ 4 t_0. 
\end{equation*} 
\end{lemma} 

\begin{proof} 
From Lemma \ref{lemma:psi-comparison}, we know that $F(t_0) \leq \psi(4t_0)$. 
Using this, we perform the substitution $x = \delta^y$, which leads to 
\begin{equation*} 
\int_0^{\delta F(t_0)}\frac{\mu_+^{-1}(x)}{x|\log x|} \, \d x 
\ \leq\ \int_0^{\delta \psi(4t_0)} \frac{\mu_+^{-1}(x)}{x|\log x|} \, \d x 
\ =\ \int_{\log_\delta( \delta \psi(4t_0) )}^\infty \frac{\mu_+^{-1}(\delta^y)}{y} \, \d y. 
\end{equation*} 
Next, let $k \in \N_0$ be such that $4t_0 \in (a_{k+1}, a_k]$, where $a_{k}$ is defined in Lemma \ref{lemma:psi-comparison}. Then we have ${\log_\delta( \delta \psi(4t_0) ) = k+1}$, and by splitting the integral into a sum of integrals, we infer 
\begin{equation*} 
\int_{k+1}^\infty \frac{\mu_+^{-1}(\delta^y)}{y} \, \d y 
\ =\ \sum_{j \geq k+1} \int_j^{j+1} \frac{\mu_+^{-1}(\delta^y)}{y} \, \dy 
\ \leq\ \sum_{j \geq k+1} \frac{\mu_+^{-1}(\delta^j)}{j}\ =\ a_{k+1}\ <\ 4t_0. 
\end{equation*} 
This completes the proof. 
\end{proof} 

\begin{proof}[Proof of Theorem \ref{Thm:explosion-Poisson}] 
The sufficiency of the integrability condition follows from Theorem \ref{Thm:sufficient-explosion-Amini}. We now turn to the necessity and prove that the integrability condition \eqref{eq:explosion-integral-cond} is indeed required for explosion. 
Since both \eqref{eq:explosion-integral-cond} and the condition on $\xi$ yielding an explosive CMJ process only depend on $\mu_+$ near zero (see Corollary \ref{cor:minimal fixed point}), 
we can assume, without loss of generality, that $\mu_+$ is convex everywhere. 
To approximate $\mu_+$, we define a sequence of cumulative mass functions $(\nu_n)_{n \in \N}$, 
where each $\nu_n$ is derived from $\mu_+$ by linearizing $\mu_+$ in the interval $[0,1/n]$. 
Formally, we set 
\begin{equation*} 
\nu_n(x)\ \defeq\ \begin{cases} 
nx\mu_+(1/n) &\text{if } x < 1/n, \\ 
\mu_+(x) &\text{if } x \geq 1/n. 
\end{cases} 
\end{equation*} 
Since $\mu_+$ is convex, we have $\nu_n \geq \mu_+$, 
and hence $\nu_n \downarrow \mu_+$ as $n \to \infty$. 
Additionally, for each $n\in\N$, $\nu_n$ is convex and satisfies the condition 
\begin{equation*} 
\sum_{k \in \N} \frac{\nu_n^{-1}(\delta^k)}{k} < \infty\quad\text{for every }\delta \in (0,1). 
\end{equation*} 
Let $\xi_n$ be a Poisson point process with intensity $\E[\xi_n[0,t]]= 1 + \nu_n(t)$, and let $F_n$ denote the (left-continuous version of the) distribution function of the explosion time in the CMJ process with reproduction point process $\xi_n$. 
By Proposition \ref{prop:Poisson-comparison} below, we know that $F_n \geq F$ (irrespective of whether $\xi$ yields an explosive CMJ process or not). 

\vspace{.1cm} 
Fix $\delta \in (0,1)$ and define $t_0 > 0$ as in Lemma \ref{lemma:psi-comparison}. 
We then have for all sufficiently large $n$ 
\begin{equation*} 
\int_0^{\delta F(t_0)} \frac{\nu_n^{-1}(x)}{x |\log x|} \, \d x 
\ \leq\ \int_0^{\delta F_n(t_0)} \frac{\nu_n^{-1}(x)}{x |\log x|} \, \d x 
\ \leq\ \frac{\delta}{2+\delta} \mu_+^{-1}(1), 
\end{equation*} 
where the last inequality follows from Lemma \ref{lemma:int-bound}, the definition of $t_0$, and the fact that $\nu_n^{-1}(1) = \mu_+^{-1}(1)$ for all sufficiently large $n$. 
Using the monotone convergence theorem, we deduce 
\begin{equation*} 
\int_0^{\delta F(t_0)} \frac{\mu^{-1}(x)}{x |\log x|} \, \d x < \infty. 
\end{equation*} 
Now, if $\xi$ yields an explosive CMJ process, then $F(t_0) > 0$ by Corollary \ref{cor:minimal fixed point}, which completes the proof. 
\end{proof} 
\subsection{Comparison methods}\label{subsec:comparison} 

In this section, we present results that enable us to infer the explosiveness of a CMJ process with a Poisson reproduction point process by comparing its intensity measure to that of another process. 
\begin{proposition}\label{prop:Poisson-comparison} 
Let $\xi$ and $\xi'$ be Poisson point processes with intensity measures $\mu$ and $\mu'$, respectively. 
Suppose there exists $t_0 > 0$ such that $\mu[0,t] \leq \mu'[0,t]$ holds for all $t \leq t_0$. 
Then, if $\xi$ yields an explosive CMJ process, so does $\xi'$. 
Moreover, the respective explosion times $T$~and $T'$ satisfy \eqref{eq:explosion-time-domination}. 
\end{proposition} 

\begin{proof} 
Let $\Smooth$ and $\Smooth'$ denote the smoothing transforms associated with $\xi$ and $\xi'$, respectively, as defined in \eqref{eq:ST}. 
By \eqref{eq:poisson-smoothing}, for all $\phi \in \cM$ and $t \leq t_0$, we have 
\begin{equation*} 
\Smooth \! \phi(t)\ =\ \exp\big( -(1-\phi) \ast \mu(t)\big)\ \geq\ \exp\big( -(1-\phi) \ast \mu'(t)\big)\ =\ \Smooth' \! \phi(t), 
\end{equation*} 
where we used that $1 - \phi$ is increasing (see the proof of Proposition \ref{prop:comparison}). 
The claim now follows from Lemma \ref{lemma:T-comparison}. 
\end{proof} 
\begin{proof}[Second proof.] 
Let $(\eta(t))_{t \geq 0}$ be a homogeneous Poisson process with unit rate. 
Then $t \mapsto \eta(\mu(t))$ is the cumulative mass function of a Poisson point process $\tilde{\xi}$ with intensity measure $\mu$. 
The same is true for $t \mapsto \eta(\mu'(t))$ so that we obtain a coupling $(\tilde{\xi}, \tilde{\xi}')$ of $\xi$ and $\xi'$. 
Now we observe that $\mu(t) \leq \mu'(t)$ for all $t \leq t_0$ implies 
\begin{equation*} 
\tilde{\xi}(t) = \eta(\mu(t)) \leq \eta(\mu'(t)) = \tilde{\xi}'(t) 
\end{equation*} 
almost surely. 
The claim follows from Proposition \ref{prop:comparison}. 
\end{proof} 
A striking feature of Theorem \ref{Thm:explosion-Poisson} is that explosion appears to be robust under scaling of $\mu_+$: the condition \eqref{eq:explosion-integral-cond} holds for $\mu_+$ if and only if it holds for $a\mu_+$ for some $a > 0$. 
In the following, we generalize this observation by replacing the convexity assumption of Theorem \ref{Thm:explosion-Poisson} with the condition 
\begin{align}\label{eq:convexity-replacement} 
\begin{gathered} 
\text{There exists a function }\varphi:[0,\infty) \to [0,\infty)\text{ with }\lim_{\lambda \to 0} \varphi(\lambda) = 0 \text{ such that}\\ 
\limsup_{(t,\lambda) \to 0}\ \frac{\mu_+(\lambda t)}{\varphi(\lambda)\mu_+(t)} < \infty. 
\end{gathered} 
\end{align} 
Note that for convex $\mu_+$, \eqref{eq:convexity-replacement} holds with $\varphi(\lambda) = \lambda$. 

\begin{lemma} 
Let $\mu_+:[0,\infty)\to [0,\infty)$ be an increasing function satisfying $\mu_+(0) = 0$, $\mu_+(t) > 0$ for all $t > 0$, and condition \eqref{eq:convexity-replacement}. 
Then, a Poisson point process with intensity $\mu=\delta_0 + \mu_+$ yields an explosive CMJ process 
if and only if a Poisson point process with intensity $\delta_0 + a \mu_+$, for some $a > 0$, does as well. 
\end{lemma} 

\begin{proof} 
By Proposition \ref{prop:Poisson-comparison}, 
it suffices to show that explosion with intensity $\delta_0 + a\mu_+$ for $a > 1$ 
implies explosion with intensity $\delta_0 + \mu_+$. 
By Lemma \ref{lemma:argument-scaling}, there exist $c > 1$ and $t_0 > 0$ 
such that 
$$ a\mu_+(t)\,\leq\,\mu_+(ct)\,\eqdef\, \mu_c(t)\quad\text{for all }t \leq t_0. $$ 
Thus, by Proposition \ref{prop:Poisson-comparison}, 
a Poisson point process with intensity $\delta_0 + \mu_c$ yields an explosive CMJ process. 
Next, we use the fact that explosion is robust under time scaling. Specifically, if 
$\xi$ yields an explosive CMJ process and $\Fbar$ is the non-trivial fixed of the smoothing transform, 
then the function 
$$ \Fbar_c(t)\,\defeq\,\Fbar(ct) $$ 
is the corresponding non-trivial fixed point of the smoothing transform associated with the scaled reproduction point process $\xi_c(B) \defeq \xi(cB)$, $B \subseteq [0,\infty)$ a Borel set. 
Therefore, $\xi_c$ yields an explosive CMJ process as well, and we thus conclude the same for a Poisson point process with intensity $\delta_0 + \mu_+$. 
\end{proof} 

\begin{corollary} 
Let $\xi$ and $\xi'$ be Poisson point processes with intensity measures $\mu$ and $\mu'$, respectively. Assume that $\mu_+'(t) \defeq \mu'(0,t]$ satisfies \eqref{eq:convexity-replacement}, and 
\begin{equation*} 
\limsup_{t \downarrow 0} \frac{\mu(0,t]}{\mu'(0,t]} < \infty. 
\end{equation*} 
Then, if $\xi$ yields an explosive CMJ process, so does $\xi'$. 
\end{corollary} 
\appendix 
\section{An auxiliary lemma} 

\begin{lemma}\label{lemma:argument-scaling} 
Let $\mu_+:[0,\infty) \to [0,\infty)$ be an unbounded increasing function with $\mu_+(0) = 0$ and $\mu_+(t) > 0$ for all $t > 0$. 
Then, the condition \eqref{eq:convexity-replacement} holds if and only if 
\begin{align} 
\begin{gathered}\label{eq:argument-scaling} 
\text{there exists }t_0 > 0 \text{ such that for all }a > 1\\ 
a\mu_+(t) \leq \mu_+(ct)\text{ for all }t \leq t_0\text{ and some }c=c(a)>1. 
\end{gathered} 
\end{align} 
\end{lemma} 

\begin{proof} 
Suppose \eqref{eq:argument-scaling} holds. 
Then there exists a function $f(s) \in (0,1)$ such that for all $s < 1$, we have ($a = 1/s$ and $f(s) = 1/c$) 
\begin{equation}\label{eq:mu_+-frac-bound} 
\frac{\mu_+(f(s) t)}{s\mu_+(t)}\ \leq\ 1\quad\text{for all }t \leq \frac{t_0}{f(s)}. 
\end{equation} 
Moreover, $\lim_{s \to 0} f(s) = 0$ must hold because if this were not true, there would exist a sequence $s_n \downarrow 0$ and some $t > 0$ such that $\mu_+(f(s_n) t)$ is bounded from below, contradicting \eqref{eq:mu_+-frac-bound}. 
It follows that the generalized inverse $f^{-1}(t) = \inf\{s> 0\,:\,f(s) \geq t\}$ exists in a neighborhood of zero and converges to $0$ as $t \to 0$. 
To see this, note that, if $0 < t < \sup_{0 < s \leq \eps} f(s)$, then $f^{-1}(t) < \varepsilon$, for any sufficiently small $\eps > 0$. 

\vspace{.1cm} 
Next, for $\lambda > 0$, define $\varphi(\lambda)$ as some $s > 0$ such that $s \leq 2f^{-1}(\lambda)$ and $f(s) \geq \lambda$. 
We then have $\varphi(\lambda) \leq 2 f^{-1}(\lambda) \to 0$ as $\lambda \to 0$ and 
\begin{equation*} 
\frac{\mu_+(\lambda t)}{\varphi(\lambda)\mu_+(t)}\ \leq\ \frac{\mu_+(f(\varphi(\lambda)) t)}{\varphi(\lambda)\mu_+(t)}\ \leq\ 1 
\end{equation*} 
by \eqref{eq:mu_+-frac-bound}, so that \eqref{eq:convexity-replacement} follows. 

\vspace{.1cm} 
Conversely, suppose that \eqref{eq:convexity-replacement} holds and that \eqref{eq:argument-scaling} is false. 
Then there exist $a > 1$, $t_0 > 0$, and a sequence $(t_n)_{n \in \N} \subseteq [0,t_0]$ such that for all $n \in \N$, 
\begin{equation}\label{eq:mu_+-t_n} 
1\ \leq\ \frac{\mu_+(n t_n)}{\mu_+(t_n)}\ <\ a. 
\end{equation} 
By compactness, there exists a subsequence $(t_{n_k})_{k \in \N}$ that converges to some limit $t \in [0,t_0]$. 
Moreover, the limit must be $0$, since otherwise there is a contradiction to \eqref{eq:mu_+-t_n}, using that $\mu_+$ increases to infinity. 
Since $\mu_+(n_k t_{n_k}) \leq a \mu_+(t_{n_k}) \to 0$ as $k \to \infty$, we infer $n_k t_{n_k} \to 0$ as well. 
Now, by \eqref{eq:convexity-replacement}, there exists a constant $C < \infty$ such that 
\begin{equation*} 
\frac{\mu_+(\lambda t)}{\varphi(\lambda)\mu_+(t)}\ <\ C 
\end{equation*} 
for all sufficiently small $\lambda, t > 0$. 
Thus, for sufficiently large $k$ (let $t = n_k t_{n_k}$, $\lambda = 1/n_k$), we have 
\begin{equation*} 
\frac{\mu_+(n_k t_{n_k})}{\mu_+(t_{n_k})}\ >\ \frac1{\varphi(1/n_k) C}, 
\end{equation*} 
which diverges to $\infty$ as $k \to \infty$ because $\lim_{\lambda \to 0}\varphi(\lambda) = 0$. 
This contradicts \eqref{eq:mu_+-t_n}, completing the proof. 
\end{proof} 



\providecommand{\bysame}{\leavevmode\hbox to3em{\hrulefill}\thinspace} 
\providecommand{\MR}{\relax\ifhmode\unskip\space\fi MR } 
\providecommand{\MRhref}[2]{%
\href{http://www.ams.org/mathscinet-getitem?mr=#1}{#2} 
} 
\providecommand{\href}[2]{#2}

\begin{acks} 
Gerold Alsmeyer acknowledges support by the Deutsche Forschungsgemeinschaft (DFG) under Germany's Excellence Strategy EXC 2044-390685587, Mathematics M\"unster: Dynamics - Geometry - Structure. 
Matthias Meiners and Jakob Stonner were supported by DFG grant ME3625/5-1. 
Part of this work was conducted during Matthias Meiners and Jakob Stonner's visit to the University of Wroc\l aw, for which they express their gratitude for the warm hospitality. 
Finally, the authors would like to thank the anonymous reviewers for their helpful remarks and for contributing a simple proof of Lemma \ref{lemma:T-comparison}. 
\end{acks}


\end{document}